\documentclass[11pt, twoside, leqno]{amsart}  
\usepackage{lipsum}
\usepackage{amsfonts}
\usepackage{graphicx}
\usepackage{epstopdf}
\usepackage{algorithmic}  
\usepackage{calligra}
\usepackage[dvipsnames]{xcolor}
\usepackage{amsfonts,amsmath,amsthm,amssymb}
\usepackage{mathtools}
\usepackage{hyperref}
\usepackage[makeroom]{cancel}
\usepackage{autonum}
\usepackage{hhline}
\usepackage{array}
\usepackage{diagbox}
\usepackage{mdframed}
\usepackage{multicol}
\usepackage{graphicx}
\usepackage{subcaption}
\usepackage{moreverb}
\usepackage{bbm}
\usepackage[margin=1.38in]{geometry}
\usepackage{todonotes}
\usepackage{scalerel,amssymb}
\allowdisplaybreaks
\usepackage{mathrsfs}  
\usepackage{lineno}
\usepackage{todonotes}
\usepackage{tikz}
 \usetikzlibrary{arrows.meta, positioning, fit, backgrounds}
\usepackage{appendix}
\usepackage{enumitem}
\usepackage{pgfplots}
\usetikzlibrary{arrows.meta}
\usepackage[numbers,sort&compress]{natbib}
\definecolor{mygreen}{HTML}{43a047}
\usepackage{subcaption}
\usepackage{doi}

\newcommand{\TE}{{\mathcal{E}}}
\newcommand{\dtau}{\textup{d}\tau}
\newcommand{\st}{{\sup_{0\le\tau\le t}}}

\newcommand{\ddt}{\frac{\textup{d}}{\textup{d}t}}

\newcommand{\dx}{\, \textup{d} x}

\newcommand{\R}{\mathbb{R}}

\hypersetup{hidelinks}

\newtheorem{theorem}{Theorem}
\newtheorem{lemma}{Lemma}
\newtheorem{proposition}{Proposition}

\numberwithin{lemma}{section}
\numberwithin{proposition}{section}
\numberwithin{theorem}{section}
\numberwithin{equation}{section}
\makeatletter
\newcommand{\leqnomode}{\tagsleft@true}
\newcommand{\reqnomode}{\tagsleft@false}
\makeatother

\definecolor{grey}{rgb}{0.5,0.5,0.5}
    \usepackage{xcolor}
\usepackage{hyperref}

\hypersetup{
    colorlinks=true,
    linkcolor=blue,
    citecolor=blue,
    urlcolor=blue
}    
\title[Navier--Stokes--Maxwell system]{Global well-posedness and Decay estimates  for solutions of the  Navier--Stokes--Maxwell system}        
     
\author[B. Said-Houari]{Belkacem Said-Houari$^\ddag$}
\thanks{$^\ddag$Department of Mathematics, College of Sciences, University of
	Sharjah, P. O. Box: 27272, Sharjah, United Arab Emirates    (\href{bhouari@sharjah.ac.ae}{bhouari@sharjah.ac.ae})}
\begin{document}
\begin{abstract} 
We investigate the global solvability and the large-time asymptotic behavior of solutions to the Navier--Stokes--Maxwell system in two and three space dimensions. The system describes the interaction between a viscous incompressible fluid and an electromagnetic field through the Lorenz force and Ohm's law. The system exhibits a parabolic--hyperbolic coupling in which the dissipative Navier--Stokes system interacts with the hyperbolic Maxwell system. The Maxwell system is damped through Ohm's law with the Maxwell correction. For sufficiently small initial data with suitable Sobolev regularity, we establish the global well-posedness of strong solutions and derive optimal decay rates for the solution and its higher-order spatial derivatives. In addition, we show that the electric field decays faster by the extra factor $(1+t)^{-1/2}$ compared to the velocity and magnetic components.  

One of the key ingredients of our analysis is a detailed study of the linearized system, which reveals refined decay properties of the electromagnetic components. These linear estimates play a crucial role in the analysis of the two-dimensional case, allowing us to overcome the logarithmic loss arising from the borderline decay estimates.  

For the nonlinear system, our approach is based on a time-weighted energy method, specifically designed to capture the distinct dissipative properties of the fluid and electromagnetic components. A suitable compensating functional is introduced to exploit the coupling between the electric and magnetic fields and thereby recover the missing dissipation of the magnetic field. Combined with careful analysis of the nonlinear terms,  these estimates yield our desired result.

					\end{abstract}   
	\vspace*{-7mm}   
	\maketitle 
	\section{Introduction}

In this paper, we consider the Navier-Stokes-Maxwell system 
\begin{subequations}\label{Main_System}
\begin{equation}\label{MHD_System}
\left\{
\begin{array}{ll}
\partial_t u+(u\cdot \nabla) u-\mu\Delta u+\nabla p=j\times B,\\[2pt]
\partial_t E-\nabla\times B=-j,\\[5pt]
\partial_t B+\nabla\times E=0,\\[5pt]
\nabla\cdot u=\nabla\cdot B=0,\\[5pt]
\sigma (E+u\times B)=j,
\end{array}
\right. 
\end{equation}
where $(x,t)$ is  in $\R^N \times \R_+$ with $(N=2,3)$. 
The system is supplemented with the initial conditions 
\begin{eqnarray}
    \label{Initial_Data}
(u, E, B)|_{t=0}=(u_0, E_0, B_0).  
\end{eqnarray}
The above system arises in the study of the flow of an electrically conducting
fluid in Plasma dynamics,
which models the motion of charged particles (ions and electrons) in an electromagnetic field. The interested reader is referred to \cite{D-book,Pai_1962,bis-book} for the derivation of \eqref{Main_System} and for some physical introduction to magnetohydrodynamics.  
Here $u, E, B: \R^N \times \R_+\to \R^3$ are vector fields defined on $\R^N $ ($N=2,3$) and represent the velocity field of an incompressible electrically conducting and non-magnetic fluid, the electric field, and the magnetic field, respectively. The first equation in \eqref{MHD_System} is the forced  Navier-Stokes equation for incompressible fluid with the Lorentz force $j\times B$.  The second equation in \eqref{MHD_System} is  Amp\`ere's law with the Maxwell correction, and the third equation is Faraday's law. 
The parameters $\mu>0$ and $\sigma>0$ denote the viscosity and the electric conductivity, respectively. 
The system exhibits a nonlinear coupling between the Navier-Stokes system and the Maxwell system through the Lorentz force   and through the electric current density $j$, which is given by Ohm's law: 
\begin{equation}\label{Ohms_Moving}
j=\sigma (E+u\times B). 
\end{equation}
\end{subequations}
Equation \eqref{Ohms_Moving}  states that the electric current is proportional to the electric field measured in a frame moving with the velocity $u$. This explains the presence of the extra term $u\times B$ in \eqref{Ohms_Moving}. 
The scalar function $p$ in \eqref{MHD_System} stands for the pressure which can be recovered from the velocity $u$ and the Lorenz force $j\times B$ via the explicit Calderon-Zygmund type operator 
\begin{equation}
 p=(-\Delta)^{-1}\big(\nabla\cdot (u\cdot\nabla)u-\nabla\cdot(j\times B)\big).
\end{equation}
System \eqref{MHD_System} is an interesting example of hyperbolic-parabolic coupling and is the hyperbolic version of the classical MHD equations.

Note that in the two-dimensional case, the unknown functions $u, E, B$ and $j$ are defined on the whole space $\R^2$ with values in $\R^3$. In this case, the operator $\nabla$ is given by 
$
\nabla=(\partial_{x_1},\partial_{x_2},0)
$
and 
\begin{equation}
\nabla\times F=(\partial_{x_2}F_3, -\partial_{x_1}F_3,\partial_{x_1}F_2-\partial_{x_2}F_2),\qquad F=(F_1, F_2, F_3). 
\end{equation}

When the electromagnetic effects are absent, namely when  $E=B=0$, 
system \eqref{Main_System}  reduces to the  classical incompressible Navier--Stokes system:
\begin{equation}\label{Navier_Stokes}
\left\{
\begin{aligned}
&\,\partial_t u+(u\cdot \nabla) u-\mu\Delta u+\nabla p=0,\\
&\,\nabla\cdot u=0,\\
&\,u|_{t=0}=u_0. 
\end{aligned}
\right. 
\end{equation}
On the other hand, in the absence of the fluid motion, that is, when $u=0$, system \eqref{Main_System} takes the form of the damped Maxwell system: 
\begin{equation}\label{Maxwell_System}
\left\{
\begin{aligned}
&\,\partial_t E+\sigma E-\nabla\times B=0,\\
&\,\partial_t B+\nabla\times E=0,\\
&\,\nabla\cdot B=0,\\
&\,(E,B)|_{t=0}=(E_0,B_0). 
\end{aligned}
\right. 
\end{equation}
 
System \eqref{Navier_Stokes} has a long and exciting mathematical history.  Since the pioneering work of Leray \cite{Leray_1934}, this system has become the focus of many mathematical research. More precisely, Leray established the existence and uniqueness of global weak solutions in two dimensions, while in three dimensions, he proved the global existence of weak solutions, with uniqueness remaining an open problem.  In addition, the solutions found by Leray are in the class 
\begin{equation}\label{Regularity_Leray}
u\in L^\infty (\R^+; L^2)\cap L^2(\R^+,\dot{H}^1) 
\end{equation}
and satisfy the energy estimate 
\begin{equation}\label{Energy_Ident_Leray}
\Vert u(t)\Vert_{L^2}^2+2\mu\int_0^t\Vert \nabla u(\tau)\Vert_{L^2}^2\dtau\leq \Vert u_0\Vert_{L^2}^2. 
\end{equation}
The global regularity of smooth solutions of \eqref{Navier_Stokes} for arbitrarily large initial data remains one of the fundamental open problems in the mathematical theory of incompressible fluid dynamics.

It is well known that system \eqref{Navier_Stokes} is invariant under the scaling 
\begin{equation}
u_\lambda(x,t)=\lambda u(\lambda x,\lambda^2 t),\qquad p_\lambda(x,t)=\lambda^2p(\lambda x,\lambda^2 t). 
\end{equation}
 More precisely, if $(u,p)$ is a  solution of  \eqref{Navier_Stokes} on $\R^N \times (0,T)$ with initial datum $u_0$,  then $(u_\lambda,p_\lambda)$ is also a solution of \eqref{Navier_Stokes} on $\R^N \times(0,T/\lambda^2) $ with the initial data $u_{0,\lambda}=\lambda u_0(\lambda x)$.
 This scaling naturally leads to the study of \eqref{Navier_Stokes} in critical spaces whose norms are invariant under the above transformation. A substantial theory of well-posedness and asymptotic behavior of solutions has been developed in these scale-invariant functional spaces. Classical example of such spaces includes $L^N(\R^N )$ \cite{Kato1984}, the homogeneous Sobolev space $\dot{H}^{N/2-1}(\R^N )$ \cite{Fujita_1964}, the  homogeneous Besov space $\dot{B}_{p,1}^{N/2-1}(\R^N )$ \cite{Cannone_1994} as well as the space $BMO^{-1}$ introduced in  \cite{Koch_Tataru_2001}.

Going back to the Navier--Stokes--Maxwell equations, 
smooth solutions of \eqref{Main_System} enjoy the following energy identity
\begin{equation}\label{Ident_Energy_1}
\frac{1}{2}\ddt \Vert (u, E, B)(t)\Vert_{L^2}^2 + \mu \Vert \nabla u (t)\Vert_{L^2}^2+ \frac{1}{\sigma}\Vert j(t)\Vert_{L^2}^2=0,  
\end{equation}
for all $t\geq 0$,  which is the only known a priori bound for  \eqref{Main_System}, so far.
It is known that the analysis of the system \eqref{Main_System} is more complicated than that of the Navier--Stokes and MHD equations. This is due to the hyperbolic nature of Maxwell's equations and the nonlinear structure of the Lorenz force. 
Identity \eqref{Ident_Energy_1} shows that the total energy of the system \eqref{Main_System} is dissipated by the viscosity and the electric resistivity. It also suggests that it might be possible to construct a weak solution, similar to the Leary weak solution for the Navier-Stokes system \eqref{Navier_Stokes} at the level of the energy \eqref{Ident_Energy_1}. However, this is still an outstanding   open problem.  
 This is due to the fact that it is hard to prove some compactness needed for the magnetic field that will guarantee the passage to the limit for the term $j\times B$ in the standard  approximating scheme for arbitrary finite energy initial data and  without assuming any additional positive regularity of the electromagnetic field.


Several results for \eqref{Main_System} in spaces close to the energy space have been established. As in the classical Navier-Stokes system, the behavior of the solution depends strongly on the spatial dimension. We therefore briefly recall the main contributions in the two-dimensional and three-dimensional settings separately. 
\subsubsection*{The 2D case} Masmoudi \cite{Masmoudi_2010} proved the existence and uniqueness of global solutions for initial data $(u_0, E_0, B_0)$ in the space 
\begin{equation}\label{Initian_data_Masmoudi}
    L^2(\R^2 )\times H^s(\R^2 )\times H^s(\R^2)\quad \text{for}\quad  s\in (0,1).
\end{equation}
  Ibrahim and Keraani \cite{Ibrahim_2011} established the  global well-posedness for sufficiently small initial data in 
  \begin{equation}\label{Assumption_IK}
      \dot{B}_2^{0,1}(\R^2 )\times L^2_{\log}(\R^2 )\times L^2_{\log}(\R^2 ),
  \end{equation}
  where the subscript ``$\mathrm{log}$'' refers to an additional logarithmic condition on the high-frequency components of the initial data and the space $L^2_{\log}(\R^2 )$ satisfies the property $\cup_{s>0} H^s(\R^2 )\subset L^2_{\log}(\R^2 ) $ with a strict inclusion. The assumptions  \eqref{Assumption_IK} in \cite{Ibrahim_2011} has been relaxed in \cite{Ger_Mess_Ibra_2014}  to initial data in the space 
  \begin{equation}
      L^2_{\log}(\R^2 )\times L^2_{\log}(\R^2 )\times L^2_{\log}(\R^2 ), 
  \end{equation}
  where the authors established a global existence result for small initial data. 
  Recently,   Ars\'enio and  Gallagher 
\cite{ag20} established the first global well posedness result for the  system \eqref{Main_System}  for initial data satisfying \eqref{Initian_data_Masmoudi}. The analysis in \cite{ag20} yields the stronger estimate 
$u\in L^2_{\rm loc} (\R_+; L^\infty(\R^2))$ istead of $u\in L^1_{\rm loc} (\R_+; L^\infty(\R^2))$ obtained previously in \cite{Masmoudi_2010}. More importantly, their estimates are uniform with respect to the speed of light, which allows them to rigorously justify the asymptotic limit towards the MHD system.  

In addition to these developments, we refer the interested reader to \cite{a19,  AHB24} for further results concerning the system \eqref{Main_System} and its inhomogeneous counterpart. 
 \subsubsection*{The 3D case}
In the three-dimensional space, the situation is considerably more delicate, and the energy bound is far from being sufficient to develop a satisfactory well-posedness theory.  Early global existence results were obtained for sufficiently small initial data in functional spaces close to the critical spaces of the Navier-Stokes equation. See for instance \cite{Ibrahim_2011,Ger_Mess_Ibra_2014}. 
More precisely, it was proved in \cite{Ibrahim_2011} that sufficiently small initial data in
\begin{equation}
\dot B^{\frac{1}{2}}_{2,1}(\R^3)\times \dot H^{\frac{1}{2}}(\R^3)\times \dot H^{\frac{1}{2}}(\mathbb R^3)
\end{equation}
give rise to a 
 unique global solution to \eqref{Main_System}. This result was later  improved in \cite{Ger_Mess_Ibra_2014}, where the Besov regularity  assumption on the velocity field was relaxed to $\dot{H}^{\frac{1}{2}}(\R^3)$, yielding global well-posedness for sufficiently small initial data in
\begin{equation}
\dot H^{\frac{1}{2}}(\R^3)\times \dot H^{\frac{1}{2}}(\R^3)\times \dot H^{\frac{1}{2}}(\mathbb R^3).
\end{equation} 
We also refer to \cite{YUE2020103071,YUE2022125747} for global well-posedness results concerning spatial classes of large initial data and to \cite{ahh24} for a recent result in the axisymmetric setting.  

Very recently, the author, together with Ibrahim and Houamed \cite{Said_IBrahim_Houamed}, studied the global well-posedness of \eqref{Main_System} for small initial data in 
\begin{equation}
\dot H^{\frac{N}{2}-1}(\R^N)\times \dot H^{\frac{N}{2}-1}(\R^N)\times \dot H^{\frac{N}{2}-1}(\mathbb R^d),\qquad N\geq 3. 
\end{equation} 
In addition, under the low-frequency assumption
\begin{equation}
    (u_0, E_0, B_0)\in \dot{B}_{2,\infty}^{-N/2}(\R^N),
\end{equation}
we showed the decay rate 
\begin{equation}
    \|(u,E,B)\|_{\dot{H}^\sigma(\R^N)}\lesssim (1+t)^{-\frac{N}{2}-\frac{\sigma}{2}},\qquad -\frac{N}{2}<\sigma\leq \frac{N}{2}-1. 
\end{equation}
Hence, the critical theory gives global existence under weak regularity assumptions and at the same time describes the large-time behavior of the solution up to the critical regularity.  
 

 The aim of this paper is to study the global well-posedness and the decay rate of solutions of \eqref{Main_System} in both two and three space dimensions. In contrast to the energy and critical-spaces theories discussed above, our purpose here is to start with a smoother solution and ask whether one can propagate this higher regularity globally and, more importantly, whether the optimal decay rate can be extended to all derivatives allowed by the initial regularity. Hence, we  
 carry out our analysis at a higher Sobolev regularity, which allows us to describe the dissipation mechanism of the system more precisely.  In particular, using the time-weighted energy method, we establish the global well-posedness of small strong solutions together with optimal decay rates for the solution and its spatial derivatives.  
  The essential difficulty is to construct a time-weighted energy norm that simultaneously captures the mismatch dissipation of the three components of the solution. A further challenge is   to recover the missing magnetic dissipation through the Maxwell system by constructing a suitable compensating functional that provides the missing dissipation of the magnetic component.
  More precisely, we prove the following decay rates: 
\begin{equation}
\begin{aligned}
    \|\nabla^k(u, E,B)(t)\|_{L^2}\lesssim &\,(1+t)^{-N/4-k/2}, \qquad 0\leq k\leq s\\ 
    \end{aligned}
\end{equation}
and 
\begin{equation}
    \|\nabla^k E(t)\|_{L^2}\lesssim \, (1+t)^{-N/4-1/2-k/2}, \qquad  0\leq k\leq s-1
\end{equation}
provided that $(u_0, E_0, B_0)\in H^s(\R^N)\cap L^1(\R^N)$. The method we use here allows us to treat the two and three-dimensional cases within a common framework.
  A particular difficulty in obtaining the above decay rates arises in two spatial dimensions, where a direct application of the $L^1$-$L^2$ decay estimates through Duhamel's formula yields a logarithmic loss. We show that this loss can be avoided by exploiting more precisely the structure of the Maxwell  equations.  In particular, the presence of the linear damping term and the form of the source term in the electric field equation yield an improved decay rate by a factor $(1+t)^{-1/2}$ of the electric field. Also, we show that the contribution to the magnetic field generated by a purely electric source enjoys a corresponding improved decay rate. These refined decay rates eliminate the logarithm loss and allow us to recover the optimal decay rate for the two-dimensional case as well.

  The remainder of the paper is organized as follows. In Section \ref{Section_Prel}, we introduce the notation and collect some preliminary estimates that will be used throughout the paper. Section \ref{Section_Linearized} is devoted to the analysis of the linearized system, where we establish the decay estimates of the solution  that will play decisive role  in the subsequent  nonlinear analysis.  In Section \ref{Section_Main_Result}, we state and discuss our main result and introduce the time-weighted energy norms. We also describe the main ideas of the proof.   In Section \ref{Section_Proof_Main_result}, we prove the main result. More precisely in Subsection \ref{Subsection_Energy_Estimate},  we derive the higher-order energy estimates. In particular, a suitable compensating functional is constructed to recover the dissipation of the magnetic field.  Subsection \ref{Subsection_Decay} is devoted to the proof of the nonlinear decay estimates. Finally, in Subsection \ref{Subsection_Closing}, we close the bootstrap  argument, thereby  completing the proof of the main result. 
 
 \section{Preliminaries}\label{Section_Prel}
 In this section, we collect some notions and results
which turn out to be useful in our proof. 

 The differential operator $\frac{\partial}{\partial{x_i}},\, 1\leq i\leq n$ will be denoted by $\partial_i$ and for a multi-index $k=(k_1,\dots, k_n), \nabla^{k}$ will be the differential operator 
\begin{equation}
\nabla^{k}=\partial_1^{k_1}\dots \partial_n^{k_n}=\frac{\partial^{|k|}}{\partial {x_1^{k_1}}\dots\partial {x_1^{k_n}}},\qquad |k|=k_1+\dots+k_n.
\end{equation}
We denote by $C, c(\epsilon),...$ some generic positive constants that may change from line to line. Also, for simplicity, we take $\sigma=1$. The notation  $A\lesssim B$ means
that there exists a positive constant $C>0$ independent of the relevant parameters, such that $A\le CB.$\smallbreak
We define the commutator $[A,B]=AB-BA$ and notice that we have 
\begin{equation}\label{Derivative_Comuta}
\nabla^{k}(AB)=[\nabla^{k},A]B+A\nabla^{k}B. 
\end{equation}

Next, we introduce the following lemma, which can be found, for example, in
\cite{Ma76,Se68}, cp. also  \cite[Lemma 7.4]{Ra92}.

\begin{lemma}
\label{Integral_lemma} Let $\alpha,\beta$ and $\gamma$  be positive constants.
If
\begin{equation}
\alpha\leq \beta,\quad  \alpha\leq \gamma+\beta-1,\,\gamma\neq 1\quad \text{or if}\quad \alpha<\beta,\, \alpha\leq \beta+\gamma-1,\,\gamma=1,
\end{equation}
then
\begin{equation}
\int_{0}^{t/2}\left( 1+t-s\right) ^{-\beta}\left( 1+s\right) ^{-\gamma}ds\leq C\left(
1+t\right) ^{-\alpha }.  \label{First_integral_inequality}
\end{equation}%
If \begin{equation}
\alpha\leq \gamma,\, \alpha\leq \gamma+\beta-1,\,\beta\neq 1\quad \text{or if}\quad \alpha<\gamma,\, \alpha\leq \beta+\gamma-1,\,\beta=1,
\end{equation} then
\begin{equation}
\int_{t/2}^{t}\left( 1+t-s\right) ^{-\beta}\left( 1+s\right) ^{-\gamma}ds\leq C\left(
1+t\right) ^{-\alpha} .
\label{Second_integral_inequality}
\end{equation}%
 \end{lemma}
The following lemma has been proved in \cite[Lemma 3.7]{St81}.

\begin{lemma}\label{Lemma_Stauss}
Let $M(t)$ be a non-negative continuous function of $t$ satisfying the inequality
\begin{equation}
M(t)\leq c_1+c_2 M(t)^{\kappa},
\end{equation}
in some interval containing $0$, where $c_1$ and $c_2$ are positive constants and $\kappa>1$. If $M(0)\leq c_1$ and
\begin{equation}
c_1c_2^{1/(\kappa-1)}<(1-1/\kappa)\kappa^{-1/(\kappa-1)},
\end{equation}
then in the same interval
\begin{equation}
M(t)<\frac{c_1}{1-1/\kappa}.
\end{equation}
\end{lemma}
   
The following product and commutator estimates are key technical tools for analyzing the nonlinear terms. The commutator estimate can be found in \cite[Lemma 4.1]{HKa06} while \eqref{First_inequaliy_Guass}
follows directly from the  Leibniz rule and  H\"older's inequality.
\begin{lemma}
\label{Guass_symbol_lemma} Let $1\leq p,\,q,\,r\leq \infty $ and $%
1/p=1/q+1/r $. Then, we have%
\begin{equation}
\Vert \nabla^{k}( uv) \Vert _{L^p}\leq C( \Vert u\Vert _{L^q}\Vert 
\nabla^{k}v\Vert _{L^r}+\Vert v\Vert _{L^q}\Vert \nabla^{k}u\Vert _{L^r}) ,\quad k\geq 0,
\label{First_inequaliy_Guass}
\end{equation}
and the commutator estimate
\begin{eqnarray}
\Vert [ \nabla^{k},f] g\Vert _{L^p}&=&\Vert \nabla^{k}(fg)-f\nabla^{k} g\Vert_{L^p}\notag\\
& \leq& C( \Vert \nabla f\Vert _{L^q}\Vert
\nabla^{k-1}g\Vert _{L^r}+\Vert g\Vert _{L^q}\Vert \nabla^{k}f\Vert _{L^r})
,\quad k\geq 1,  \label{Second_inequality_Gauss}
\end{eqnarray}
 for some constant $C>0$.
\end{lemma}
We will also make use of the following Gagliardo--Nirenberg interpolation inequality.
\begin{lemma}
(\cite{Ner59}) \label{interpolation_lemma}  Let $1\leq
p,\,q\,,r\leq \infty $, and let $m$ be a positive integer. Then for any
integer $j$ with $0\leq j< m$, we have%
\begin{equation}
\left\Vert \nabla ^{j}u\right\Vert _{L^{p}}\leq C\left\Vert
\nabla^{m}u\right\Vert _{L^{r}}^{\alpha}\left\Vert u\right\Vert
_{L^{q}}^{1-\alpha}  \label{Interpolation_inequality}
\end{equation}%
where
\begin{equation}
\frac{1}{p}=\frac{j}{N}+\alpha\left( \frac{1}{r}-\frac{m}{N}\right) + \frac{%
1-\alpha}{q}
\end{equation}%
for $\alpha$ satisfying $j/m\leq \alpha \leq 1$ and $C$ is a positive
constant depending only on $n,\, m,\,j,\,q,\,r$ and $\alpha$.
 There
are the following exceptional cases:
\begin{enumerate}
\item If $j=0,\,rm<n$ and $q=\infty $, then we made the additional
assumption that either $u(x)\rightarrow 0$ as $|x|\rightarrow \infty $ or $%
u\in L^{q^{\prime }}$ for some $0<q^{\prime }<\infty .$

\item If $1<r<\infty $ and $m-j-N/r$ is a nonnegative integer, then (\ref{Interpolation_inequality}) holds
only for $j/m\leq \alpha< 1$.
\end{enumerate}
\end{lemma}

We next establish an elementary convolution inequality that will be used in the subsequent analysis.    
\begin{lemma}\label{Lemma_Integral_Inequality}
    Let $c>0$ and $\rho(|\xi|)=|\xi|^2/(1+|\xi|^2)$. There exists a small constant $0 <c_1<c$ such that 
    \begin{equation}\label{Ineq_integral}
        {\rm{K}}=\int_0^t e^{-(t-\tau) }e^{-c\rho(|\xi|)\tau }\dtau \lesssim e^{-c_1\rho(|\xi|)t}
    \end{equation}
    for all $t\geq 0$. 
\end{lemma}
\begin{proof}
    To show \eqref{Ineq_integral}, we split the integral $\rm{K}$ as  
    \begin{equation}
    \begin{aligned}
      {\rm{K}}=&\, \int_0^{t/2} e^{-(t-\tau) }e^{-c\rho(\xi)\tau }\dtau+\int_{t/2}^{t}e^{-(t-\tau) }e^{-c\rho(\xi)\tau }\dtau\\
      =&\,\rm{K}_1+\rm{K}_2. 
       \end{aligned}
    \end{equation}
    To estimate $\rm{K}_1$, we have by using the inequality $e^{-(t-\tau)}\leq e^{-t/2}$ which is valid for $0\leq \tau\leq t/2$, together with the fact that $0\leq \rho(|\xi|)\leq 1$, 
    \begin{equation}
        {\rm{K}_1}\leq \frac{t}{2}e^{-t/2}\lesssim e^{-c_1 \rho(|\xi|)t}, 
    \end{equation}
    for sufficiently small $c_1>0$. 

    On the other hand, for $\rm{K}_2$, we have for $t/2\leq \tau<t$ the inequality $e^{-c\rho(|\xi|)\tau}\leq e^{-\frac{c}{2}\rho(|\xi|)t }$ and hence 
    \begin{equation}
        {\rm{K}_2}\lesssim e^{-\frac{c}{2}\rho(|\xi|)t }\int_{t/2}^t e^{-(t-\tau)}\dtau\lesssim e^{-\frac{c}{2}\rho(|\xi|)t }. 
    \end{equation}
    Hence, the estimate \eqref{Ineq_integral} holds by collecting the above two estimates of $\rm{K}_1$ and $\rm{K}_2$. 
\end{proof}
\section{The linearized system }\label{Section_Linearized}
The analysis of the linearized system is relevant to understand the behavior of the nonlinear system. 
In this section, we consider the linearized system associated with \eqref{Main_System} around the trivial solution, which decouples the fluid and the electromagnetic components: 
\begin{equation}\label{MHD_System_Linearized_With_Pressure}
\left\{
\begin{array}{ll}
\partial_t u-\mu\Delta u+\nabla p=0,\\[3pt]
\partial_t E-\nabla\times B=-E,\vspace{0.1cm}\\[3pt]
\partial_t B+\nabla\times E=0,\vspace{0.1cm}\\[3pt]
\nabla\cdot u=\nabla\cdot B=0,\vspace{0.1cm}
\end{array}
\right. 
\end{equation}
The velocity is governed by the linear Stokes system, while the electric and magnetic fields satisfy the coupled damped Maxwell system.  
We use the energy method in the Fourier space to derive pointwise estimates of \eqref{MHD_System_Linearized_With_Pressure} that give the decay rate in the energy space of the solution. To achieve this, we rely on the Lyapunov functional method. The advantage of this method is to overcome the heavy computation of the eigenvalues and eigenvectors of the linearized problem, and it can be easily combined with a higher regularity energy-type method to study the nonlinear problem.

Recall the Helmotz-Hodge decomposition of a vector field $u$ as 
\begin{equation}
u=\nabla p+v,\qquad \text{with}\qquad \nabla\cdot v=0
\end{equation}


Applying Leray projection to the first equation
in \eqref{MHD_System_Linearized_With_Pressure}, we get 
\begin{equation}\label{MHD_System_Linearized}
\left\{
\begin{array}{ll}
\partial_t u-\mu\Delta u=0,\\[3pt]
\partial_t E-\nabla\times B=-E,\\[3pt]
\partial_t B+\nabla\times E=0,\\[3pt]
\nabla\cdot u=\nabla\cdot B=0,\\[3pt]
\end{array}
\right. 
\end{equation}
Applying the Fourier transform to \eqref{MHD_System_Linearized}, we obtain the system  
\begin{equation}\label{MHD_System_Fourier}
\left\{
\begin{array}{ll}
\partial_t \hat{u}+\mu|\xi|^2 \hat{u}=0,\\[3pt]
\partial_t \hat{E}-i\xi\times \hat{B}=-\sigma \hat{E},\\[3pt]
\partial_t \hat{B}+i\xi\times \hat{E}=0,\\[3pt]
i\xi\cdot \hat{u}=i\xi\cdot \hat{B}=0,\\[3pt]
\end{array}
\right. 
\end{equation}
with the initial condition in the Fourier space 
\begin{equation}\label{Initial_Data_Fourier}
\hat{u} (\xi,0)=\hat{u}_0(\xi),\qquad \hat{E}(\xi,0)=\hat{E}_0(\xi),\qquad \hat{B}(\xi,0)=\hat{B}_0(\xi). 
\end{equation}

The first equation in \eqref{MHD_System_Fourier} is decoupled from the other two equations and it is the heat equation and therefore we have 
\begin{equation}
\hat{u}(\xi,t)=e^{-\mu|\xi|^2}\hat{u}_0(x). 
\end{equation}
Hence, we get 
\begin{equation}
u(x,t)=G(x,t)\ast u_0(x),
\end{equation}
where $G(x,t)$ is the heat kernel given by 
\begin{equation}
G\left( t,x\right) =\left( 4\pi \mu t\right) ^{-N/2}e^{-\left\vert x\right\vert
^{2}/(4\mu t)}.  \label{heat_kernel}
\end{equation}%
Therefore, for any $k=1,2,\dots$, and $1\leq p\leq \infty $, we have (see \cite[Estimate (1.8)]{Giga_Book})
\begin{equation}
\Vert \nabla  ^{k}u\left( t\right) \Vert _{L^p}\leq
Ct^{-\alpha -k/2}\left\Vert u_{0}\right\Vert _{L^1
},\qquad \alpha =\left( N/2\right) \left(
1-1/p\right)
 \label{Main_parabolic_estimate}
\end{equation}%
where $C$ is a positive constant.

\subsection{Decay estimates for the linearized Maxwell system}
In this section, we establish the decay rate of the linearized Maxwell system. First, we use the energy method in the Fourier space to construct a compensating functional which allows to recover the dissipation of the magnetic field despite the absence of the direct dissipation in its equation. We then combine this compensating functional  with the   energy to build  
an appropriate Lyapunov functional that is equivalent to the energy and captures the dissipation of both components. This process yields the estimate \begin{equation}\label{Pointwise_E_B}
    |(\hat{E}, \hat{B})(\xi,t)|\lesssim e^{-c\frac{|\xi|^2}{1+|\xi|^2}t}|(\hat{E}_0, \hat{B}_0)(\xi)|,\quad c>0, 
\end{equation} 
as shown in Lemma \ref{Lemma_L} below. With estimate \eqref{Pointwise_E_B} at hand, we then combine  Plancherel's identity together with a suitable frequency decomposition to derive the decay estimate for  $\|\nabla^k(E, B)\|_{L^2}$ in the physical space as stated in Lemma \ref{Decay_Linear_Theor} below.   
We subsequently refine these estimates by exploiting the specific structure of the Maxwell system to obtain faster decay for the electric field and purely electric initial data, for the corresponding magnetic field.  

We begin by  considering  the linearized Maxwell system 
\begin{equation}\label{Maxwell_System_Fourier}
\left\{
\begin{array}{ll}
\partial_t \hat{E}+\hat{E}=i\xi\times \hat{B},\\[3pt]
\partial_t \hat{B}+i\xi\times \hat{E}=0,\\[3pt]
i\xi\cdot B=0,\\[3pt]
\hat{E}(\xi,0)=\hat{E}_0(\xi),\qquad \hat{B}(\xi,0)=\hat{B}_0(\xi).
\end{array}
\right. 
\end{equation}
We want to perform  Fourier analysis on \eqref{Maxwell_System_Fourier} and build an appropriate Lyapunov functional that will give us an estimate of the solution $\hat{V}=(\hat{E},\hat{B})$ in terms of the initial data $\hat{V}_0=(\hat{E}_0,\hat{B}_0)$ of the form 
\begin{equation}\label{Estimate_V}
|\hat{V}(\xi,t)|\leq Ce^{-c\rho(|\xi|) t}|\hat{V}_0(\xi)|,
\end{equation}
where $\rho(|\xi|)$ is a function of $|\xi|$ that will be determined later on.  

First, let us denote by $\langle.,. \rangle $ the dot product in $\mathbb{C}^N$. We define the energy function of \eqref{Maxwell_System_Fourier} as
\begin{equation}\label{Energy}
\hat{\mathscr{E}}(\xi,t):=|\hat{E}(\xi,t)|^2+|\hat{B}(\xi,t)|^2=|\hat{V}(\xi,t)|^2.
\end{equation}
Wee have the following lemma.
\begin{lemma}\label{Lemma_dissipation}
Let $\hat{V}(\xi,t)=(\hat{E},\hat{B})(\xi,t)$ be the solution of \eqref{Maxwell_System_Fourier}, then we have for all $t\geq 0$, 
\begin{equation}\label{Energy_Fourier_dissipation}
\frac{1}{2}\ddt\hat{\mathscr{E}}(\xi,t)=- |\hat{E}(\xi,t)|^2.
\end{equation}
\end{lemma}
\begin{proof}
Taking the dot product in $\mathbb{C}^n$ of the first equation with $\hat{E}$ and of the second equation with $\hat{B}$, taking the real part and adding the resulting equations, then \eqref{Energy_Fourier_dissipation} holds true. 
\end{proof}

It is clear that the energy 
 identity \eqref{Energy_Fourier_dissipation} does not provide any direct dissipation on the magnetic component  
$\hat{B}$.  However, the
interaction of the dissipative  component $\hat{E}$ with the time-dynamics generated by the coupled system 
\eqref{Maxwell_System_Fourier} 
allows the dissipation of $\hat{E}$ to to transferred to  $\hat{B}$.  To capture this effect, we introduce the compensating functional  $F(\xi,t)$ define as 
\begin{equation}\label{Functional_F}
F(\xi,t):=-\mathrm{Re}\langle i\xi\times \hat{E},\hat{B} \rangle. 
\end{equation}
Hence,  we have the following lemma. 
\begin{lemma}\label{Lemma_F}
Let $\hat{V}(\xi,t)=(\hat{E},\hat{B})(\xi,t)$   be the solution  of \eqref{Maxwell_System_Fourier}.  For any positive constant $0<\epsilon<1$, there exists a positive constant $c(\epsilon)$   such that for all $t\geq 0$, 
\begin{equation}\label{B_dissipation}
\ddt F(\xi,t)+(1-\epsilon)|\xi|^2|\hat{B}|^2\leq c(\epsilon)(1+|\xi|^2)|\hat{E}|^2.
\end{equation}

\end{lemma}
\begin{proof}
Taking the curl of the second equation in \eqref{MHD_System_Linearized}, and then applying the Fourier transform, we obtain  
\begin{equation}\label{curl_2}
\partial_t (i\xi\times \hat{E})-i\xi\times(i\xi\times \hat{B})=- i\xi\times\hat{E}.
\end{equation}
Taking the dot product of \eqref{curl_2} with $-\hat{B}$, we get 
\begin{eqnarray*}
-\ddt \langle  (i\xi\times \hat{E}),\hat{B}\rangle+\langle  (i\xi\times \hat{E}),\hat{B}_t\rangle+\langle i\xi\times(i\xi\times \hat{B}), \hat{B}\rangle=\langle i\xi\times\hat{E},\hat{B} \rangle. 
\end{eqnarray*}
Since $i\xi\cdot \hat{B}=0$, then we have $\xi\times(\xi\times \hat{B})=-|\xi|^2\hat{B}$. Using this last identity together with the second equation in \eqref{Maxwell_System_Fourier}, and taking the real part, we obtain  
\begin{eqnarray}\label{F_dt}
\ddt F(\xi,t)+|\xi|^2|\hat{B}|^2=|\xi\times \hat{E}|^2+\mathrm{Re}\langle i\xi\times\hat{E},\hat{B} \rangle. 
\end{eqnarray}
Now, Cauchy-Schwarz inequality together with Young's inequality, give, for  any $\epsilon>0$, 
\begin{equation}
|\mathrm{Re}\langle i\xi\times\hat{E},\hat{B} \rangle|\leq \epsilon |\xi|^2|\hat{B}|^2+c(\epsilon)|\hat{E}|^2.
\end{equation}
Plugging the above estimates into \eqref{F_dt}, we get \eqref{B_dissipation} which finishes the proof of Lemma \ref{Lemma_F}. 
\end{proof}
In the following lemma, we prove a pointwise estimate for the solution of \eqref{Maxwell_System_Fourier}. 
\begin{lemma}\label{Lemma_L}
Let $\hat{V}(\xi,t)=(\hat{E},\hat{B})(\xi,t)$ be the solution of \eqref{Maxwell_System_Fourier}, then we have for all $t\geq 0$, 
\begin{equation}\label{Estimate_Poitwise}
|\hat{V}(\xi,t)|\leq \frac{d_3}{d_2} e^{-\frac{2d_1}{d_3} \rho(|\xi|)t}|\hat{V}(\xi,0)|,
\end{equation}
where $d_1,d_2$ and $d_3$ are positive constants given below and 
\begin{equation}
\rho(|\xi|)=\frac{|\xi|^2}{1+|\xi|^2}. 
\end{equation} 
\end{lemma}
\begin{proof}
We define the Lyapunov functional $L(\xi,t)$ as
\begin{equation}\label{Lyapunov_F}
L(\xi,t):=\gamma (1+|\xi|^2)\mathscr{\hat{E}}(\xi,t)+F(\xi,t), 
\end{equation}
 where $\gamma$ is a positive constant that will be fixed later on. Taking the derivative of \eqref{Lyapunov_F} with respect to $t$ and exploiting \eqref{Energy_Fourier_dissipation} and \eqref{B_dissipation}, we obtain, for all $t\geq 0$,  
 \begin{eqnarray}\label{L_dt}
\ddt L(\xi,t)+(\gamma-c(\epsilon))(1+|\xi|^2)|\hat{E}|^2+(1-\epsilon)|\xi|^2|\hat{B}|^2\leq 0. 
\end{eqnarray}
Choosing $\epsilon<1$ and once $\epsilon$ is fixed, we take $\gamma$ large enough such that $\gamma>c(\epsilon)$, then we deduce that there exists a constant  $d_1>0$, such that, for $t\geq 0$, we have  
\begin{eqnarray}\label{L_dt_1}
\ddt L(\xi,t)+2d_1|\xi|^2\hat{\mathscr{E}}(\xi,t)\leq 0,
\end{eqnarray}
with $d_1=\min(1-\varepsilon,\gamma-c(\varepsilon))$. On the other hand, it is not hard to see that for $\gamma$ large enough, there exist two positive constants $d_2$ and $d_3$ such that, for all $t\geq 0$, 
\begin{equation}\label{Equi_L_E}
d_2(1+|\xi|^2)\hat{\mathscr{E}}(\xi,t)\leq L(\xi,t)\leq d_3(1+|\xi|^2)\hat{\mathscr{E}}(\xi,t). 
\end{equation}
 Comparing \eqref{L_dt_1} and \eqref{Equi_L_E}, we obtain  
 \begin{equation}
\ddt L(\xi,t)+\frac{2d_1}{d_3}\frac{|\xi|^2}{1+|\xi|^2}L(\xi,t)\leq 0. 
\end{equation}
Integrating the above estimate with respect to $t$, we obtain  for all $t\geq 0$, 
\begin{eqnarray}\label{L_dt_2}
L(\xi,t)\leq L(\xi,0)e^{-\frac{2d_1}{d_3}\frac{|\xi|^2}{1+|\xi|^2}t}.
\end{eqnarray}
Using \eqref{Equi_L_E} together with \eqref{Energy}, we deduce that  
\begin{equation}
    \hat{\mathscr E}(\xi,t)\leq \frac{d_3}{d_2}\hat{\mathscr E}(\xi,0)e^{-\frac{2d_1}{d_3}\rho(|\xi|)t}.
\end{equation}
Recalling \eqref{Energy},  then \eqref{Estimate_Poitwise} is fulfilled. 
This ends the proof of Lemma \ref{Lemma_L}. 
\end{proof}
\begin{lemma}\label{Decay_Linear_Theor} There exist two positive constants $C$ and $c$ such that the solution  $V=(E,B)$ of \eqref{Maxwell_System_Fourier} with initial data $V_0$  in $H^s(\R^N )\cap L^{1}(\R^N )$ satisfies  
\begin{equation}\label{Main_Estimate_Linear}
\Vert \nabla^k V(t)\Vert_{L^2}\leq C(1+t)^{-N/4-k/2}\Vert V_0\Vert_{L^1}+e^{-\frac{c}{2}t} \Vert \nabla^k V_0\Vert_{L^2},\qquad 0\leq k\leq s.
\end{equation}
\end{lemma}
\begin{proof}
To prove Lemma \ref{Decay_Linear_Theor}, we have by Plancherel  theorem and the estimate (\ref{Estimate_Poitwise}) that 
\begin{equation}
\begin{aligned}
    \label{Pranch_Identity}
\Vert \nabla^{k}V(t)\Vert _{L^{2}}^{2}=&\,\int_{\mathbb{R}
}\left\vert \xi \right\vert ^{2k}\vert \hat{V}(\xi ,t)\vert^{2}\textup{d}\xi\notag\\
\leq &\,\frac{d_3}{d_2}\int_{\mathbb{R}^N
}\left\vert \xi \right\vert ^{2k}e^{-\frac{2d_1}{d_3}\rho(|\xi|)t}\vert \hat{V}(\xi,0)\vert^{2}\textup{d}\xi.
\end{aligned}
\end{equation}
It is obvious that the term on the right-hand side of (\ref{Pranch_Identity}) depends on the behavior of the function $\rho(|\xi|)$. Since
\begin{equation}\label{rho_behavior}
\rho(\xi)\geq\left\{
\begin{array}{ll}
\frac{1}{2}|\xi|^2,& \text{for } |\xi|\leq 1, \\[4pt]
\frac{1}{2}, & \text{for } |\xi|\geq 1,
\end{array}%
\right.
\end{equation}%
then,  we write the integral on the right-hand side of (\ref{Pranch_Identity}) as
\begin{equation}
\begin{aligned}
\int_{\mathbb{R}^N
}\left\vert \xi \right\vert ^{2k}e^{-c\rho(|\xi|)t}\vert \hat{V}_0(\xi)\vert^{2}\textup{d}\xi =
&\,\int_{\left\vert \xi \right\vert \leq 1}\left\vert \xi \right\vert
^{2k}e^{-\frac{2d_1}{d_3}\rho (|\xi| )t}\vert \hat{V}(\xi ,0)\vert ^{2}\textup{d}\xi\notag\\
&+\int_{\left\vert \xi \right\vert \geq 1}C\left\vert \xi \right\vert
^{2k}e^{-\frac{2d_1}{d_3}\rho(|\xi| )t}\vert \hat{V}(\xi ,0)\vert ^{2}\textup{d}\xi
\label{inequality L1+L2}\notag \\
:= &\,L_{1}+L_{2}.  \label{L_1_L_2_estimate}
\end{aligned}
\end{equation}
Concerning the integral $L_1$, we have
\begin{equation}
\begin{aligned}
L_{1}
\leq &\,C\sup_{\vert \xi \vert \leq 1}\left\{ \vert \hat{V}
(\xi ,0)\vert ^{2}\right\} \int_{\left\vert \xi \right\vert \leq
1}\left\vert \xi \right\vert ^{2k}e^{-c|\xi |{^2}t}
\textup{d}\xi \nonumber \\
\leq &\,C\Vert \hat{V}_{0}(t)\Vert _{L^{\infty
}}^{2}\int_{\left\vert \xi \right\vert \leq 1}\left\vert \xi \right\vert
^{2k}e^{-c|\xi|
{^2}%
t} \textup{d}\xi \nonumber
\end{aligned}
\end{equation}
Passing to the polar coordinates and using the inequality
\begin{equation}\label{Inequality_exponential}
\int_{0}^{1}\left\vert \xi \right\vert ^{\sigma }e^{-c|\xi| ^{2}t}d|\xi| \leq
C\left( 1+t\right) ^{-\left( \sigma +N\right) /2},
\end{equation}
we deduce that
\begin{eqnarray}\label{L_1_inequality}
L_{1} \leq  C(1+t)^{-N/2-k}\left\Vert V_0\right\Vert _{L^{1
}}^{2}.
\end{eqnarray}
On the other hand, we have
\begin{equation}
\begin{aligned}
    \label{L_2_estimate}
L_{2}
\leq &\,Ce^{-\frac{1}{2}t} \int_{\left\vert \xi \right\vert \geq 1}\left\vert \xi
\right\vert ^{2k}\vert \hat{V}(\xi ,0)\vert
^{2}\textup{d}\xi\notag \\
\leq &\,Ce^{-\frac{1}{2}t}\Vert \nabla  ^{k}V_0\Vert _{L^{2}}^{2}.
\end{aligned}
\end{equation}
Consequently, the estimate (\ref{Main_Estimate_Linear}) follows by combining, \eqref{L_1_inequality} and \eqref{L_2_estimate}. Thus the proof of Lemma  \ref{Decay_Linear_Theor} is finished. 
\end{proof}

Taking \eqref{Main_parabolic_estimate} into account, we immediately get the linear estimate for the whole linear system \eqref{MHD_System_Linearized_With_Pressure} .
\begin{lemma}\label{MHD_linear_stability}
	Let $s\in\R$. There exist two positive constants $C$ and $c$ such that the solution $ U=(u,E,B) $ of \eqref{MHD_System_Linearized_With_Pressure} with initial data  $ V_0=(E_0, B_0)\in H^s(\R^N )\cap L^1(\R^N ) $ and $ u_0\in L^1(\R^N ) $ satsfies for any $ 0\le k \le s $,  the following estimate
	\begin{equation}\label{MHD_linear_estimate}
		\|\nabla^k U(t)\|_{L^2}\le C(1+t)^{-N/4-k/2}\|U_0\|_{L^1}+e^{-ct}\|\nabla^k V_0\|_{L^2}.
	\end{equation}
\end{lemma}
\subsubsection{A refined decay estimate for the electric field $E$} In three dimension,  the linear decay rates of $U$ obtained in Lemma \ref{MHD_linear_stability} persist for the nonlinear problem. However, the situation is more delicate in the two-dimensional case where the decay rate in \eqref{MHD_linear_estimate} becomes a borderline for some nonlinear interactions. In particular, when estimating the term $E\times B$ in the nonlinear energy estimate, using only the decay rate of $U$ leads to a logarithm loss. To avoid this loss,  we derive a faster decay rate for the electric field component $E$. This is possible due to the dissipative  structure of the Maxwell equation. At the linear level, this is clear from the linearized equation: 
\begin{equation}\label{Linearized_E}
    \partial_t E+E=\nabla\times B,
\end{equation}
    where the  particular form of the source term in \eqref{Linearized_E} will allow us to gain a factor $|\xi|$ which yields and improved decay estimate for $E$ by a factor $(1+t)^{-1/2}$. This improvement is crucial to recover the decay rate in two dimension. 
\begin{lemma}\label{Lemma_E}
    Under the assumptions of Lemma \ref{MHD_linear_stability}, it holds that 
\begin{equation}\label{MHD_linear_estimate_E}
		\|\nabla^k E(t)\|_{L^2}\le C(1+t)^{-N/4-1/2-k/2}\|V_0\|_{L^1}+e^{-ct}\|\nabla^k V_0\|_{L^2}.
	\end{equation}
\end{lemma}
\begin{proof}
    Taking the Fourier transform of \eqref{Linearized_E}, we obtain 
    \begin{equation}\label{Fourier_E}
\partial_t\hat{E}+\hat{E}=i\xi\times \hat{B}     
    \end{equation}
    Solving \eqref{Fourier_E}, we obtain 
    \begin{equation}\label{E_Equation_integral}
        \hat{E}(t,\xi)=e^{-t}\hat{E}_0(\xi)+\int_0^t e^{-(t-\tau)}i\xi\times \hat{B}(\tau,\xi)\dtau. 
    \end{equation}
    Using \eqref{Estimate_V}, we have 
    \begin{equation}\label{B_estimate}
        |\hat{B}(\tau,\xi)|\lesssim e^{-c\rho(|\xi|) \tau}|\hat{V}_0(\xi)|,\qquad \rho(|\xi|)=\frac{|\xi|^2}{1+|\xi|^2}.
    \end{equation}
    Hence,  plugging \eqref{B_estimate} into \eqref{E_Equation_integral}, we obtain 
    \begin{equation}\label{E_Equation_integral_2}
        |\hat{E}(t,\xi)|=e^{-t}|\hat{E}_0(\xi)|+\int_0^t e^{-(t-\tau)}|\xi| e^{-c\rho(|\xi|) \tau}|\hat{V}_0(\xi)|\dtau. 
    \end{equation}
    Using the inequality (see Lemma \ref{Lemma_Integral_Inequality})
    \begin{equation}
        \int_0^t e^{-(t-\tau)}e^{-c\rho(|\xi|)\tau}\dtau\lesssim e^{-c_1\rho(|\xi|)t},\quad c_1>0,
    \end{equation}
    we obtain 
\begin{equation}\label{E_Equation_integral_2}
        |\hat{E}(t,\xi)|^2\lesssim\, e^{-2t}|\hat{E}_0(\xi)|^2+|\xi|^2 e^{-2c_1\rho(|\xi|) t}|\hat{V}_0(\xi)|^2. 
    \end{equation}
    Hence, using the Plancherel  theorem, we obtain 
    \begin{equation}
        \begin{aligned}
            \|\nabla^k E(t)\|_{L^2}^2\lesssim e^{-2t}\|\nabla^k E\|_{L^2}^2 + \int_{\R^N }|\xi|^{2k+2}e^{-2c_1\rho(|\xi|) t}|\hat{V}_0(\xi)|^2\textup{d}\xi
        \end{aligned}
    \end{equation}
    Following the same steps as in the proof of Lemma \ref{Main_Estimate_Linear}, the estimate \eqref{MHD_linear_estimate_E} holds. We omit the details. 
\end{proof}
\subsubsection{A refined decay estimate for the magnetic field $B$}
In this section, we show that the decay rate of $B$ in \eqref{Main_Estimate_Linear} can be improved when $V_0=(E_0, 0)$. In this case, we gain an extra factor of $|\xi|$ in the estimate of $\hat{B}(t,\xi)$ which eventually yields an extra decay of $(1+t)^{-1/2}$ for the norm $\|\nabla^kB(t)\|_{L^2}$. More precisely, we have the following lemma. 
\begin{lemma}\label{Improved_Estimate_B_Lemma}
    Assume that $V_0=(E_0, 0)$, then it holds that 
    \begin{equation}\label{Improved_Estimate_B}
        \|\nabla^k B(t)\|_{L^2}\lesssim (1+t)^{-N/4-1/2-k/2}\|E_0\|_{L^1}+e^{-ct}\|\nabla^k E_0\|_{L^2}. 
    \end{equation}
\end{lemma}
The importance of the estimate \eqref{Improved_Estimate_B} will be clear in the nonlinear analysis. Without the faster decay \eqref{Improved_Estimate_B}, we will also get a logarithm loss when estimating the nonlinear contribution associated with $f_4=(0,-u\times B,0)$. However, since this term acts only on the electric-field equation, then  Lemma \ref{Improved_Estimate_B_Lemma} yields an additional factor $(1+t)^{-1/2}$ in the decay of $B$. This refinement is enough to remove the logarithmic loss and recover the optimal decay rate. 

\begin{proof}[Proof of Lemma \ref{Improved_Estimate_B_Lemma}]
    Since $B_0=0$, the magnetic field satisfies the damped wave equation 
    \begin{equation}
        B_{tt}-\Delta B+B_t=0,\qquad B(t=0)=0,\qquad B_t(t=0)=-\nabla\times E.
    \end{equation}
    Taking the Fourier transform of the above equation, we obtain 
    \begin{equation}
        \hat{B}_{tt} +\hat{B}_t+|\xi|^2 \hat{B}=0,\qquad \hat{B}(t=0)=0,\qquad \hat{B}_t(t=0)=-i\xi\times \hat{E}_0.
    \end{equation}
    Hence, the solution is 
    \begin{equation}
       \hat{B}(t,\xi)=-\hat{K}_2(t,\xi) (i\xi\times \hat{E}),\quad \text{with}\quad  \hat{K}_2(t,\xi)=\frac{e^{\lambda_+(\xi)t}-e^{\lambda_-(\xi)t}}{\lambda_+-\lambda_-}\quad 
    \end{equation}
    is the damped wave propagator and 
    \begin{equation}
\lambda_\pm=\frac{-1\pm\sqrt{1-4|\xi|^2}}{2}. 
    \end{equation}
    Arguing as in the low,  intermediate  and high frequency analysis  in   \cite[Proposition 4.1]{Ka2000}  and \cite[Lemma 1] {Ma76} one obtains 
    \begin{equation}
    |\hat{K}(t,\xi)|\lesssim \left\{
        \begin{aligned}
          & e^{-c|\xi|^2 t},& |\xi|\leq 1,\\
          & |\xi|^{-1}e^{-ct},& |\xi|\geq 1.
        \end{aligned}
        \right.
    \end{equation}
    Hence, we obtain for $|\xi|\leq 1$
    \begin{equation}
        |\hat{B}(t,\xi)|\lesssim e^{-c|\xi|^2 t}|i\xi\times \hat{E}_0|
    \end{equation}
    and for $|\xi|\geq 1$, we have 
    \begin{equation}
    \begin{aligned}
        |\hat{B}(t,\xi)|\lesssim&\, |\xi|^{-1}e^{-c t}|i\xi\times \hat{E}_0|\\
        \lesssim&\,e^{-ct}|\hat{E}_0(\xi)|. 
        \end{aligned}
    \end{equation}
    Collecting the above two estimates, we arrive at:
    \begin{equation}\label{Estimate_Freq_B}
    |\hat{B}(t,\xi)\lesssim |\left\{
        \begin{aligned}
          & |\xi|e^{-c|\xi|^2 t}|\hat{E}_0(\xi)|,& |\xi|\leq 1\\
          & e^{-ct}|\hat{E}_0(\xi)|,& |\xi|\geq 1.
        \end{aligned}
        \right.
    \end{equation}
    Now,   we use the Plancherel identity to write 
    \begin{equation}
\Vert \nabla^{k}B(t)\Vert _{L^{2}}^{2}=\int_{\mathbb{R}
}\left\vert \xi \right\vert ^{2k}\vert \hat{B}(\xi ,t)\vert^{2}\textup{d}\xi. 
\end{equation}
Then, we split the integral into low and high-frequencies and we use \eqref{Estimate_Freq_B} and arguing as in the proof of Lemma \ref{Decay_Linear_Theor}, yields the desired result. 
\end{proof}
\section{The nonlinear system--statement of the main results}\label{Section_Main_Result}
In this section, we analyze the nonlinear system 
 \eqref{Main_System}. Our main goal is to establish global time estimates for small initial data and to derive the optimal decay rates of the solution.  
 In particular, we show that the nonlinear solution exhibits the same algebraic decay as the corresponding linearized evolution which shows that the long-time behavior is governed by the linear dynamics. This is achieved through a combination of a time-weighted energy method with the linear decay estimates established in Section \ref{Section_Linearized} and a careful analysis of the nonlinear terms through Duhamel's formula. 
 
 Applying  the projection $\mathbb{P}$ to system \eqref{MHD_System}, we get, 
\begin{equation}\label{MHD_System_Nonl}
\left\{
\begin{array}{ll}
\partial_t u-\mu\Delta u=\mathbb{P}(- E\times B- (u\times B)\times B)-\mathbb{P}((u\cdot \nabla) u),\vspace{0.2cm}\\
\partial_t E-\nabla\times B+E=-u\times B,\vspace{0.2cm}\\
\partial_t B+\nabla\times E=0,\vspace{0.2cm}\\
\nabla\cdot u=\nabla\cdot B=0.\end{array}
\right. 
\end{equation}
The above system can be recast as 
\begin{equation}
U_t+\mathcal {L}U=G, 
\end{equation}
where $U=(u,E,B)^T, \, G=(-\mathbb{P}( E\times B+ (u\times B)\times B)-\mathbb{P}(u\cdot \nabla u),-u\times B,0)$ and 
\begin{equation}
\mathcal {L}:=
\left(
\begin{array}{ccc}
 -\mu\Delta & 0  &  0 \\
  0&\mathrm{Id}   & -\nabla\times  \\
  0& \nabla\times  &  0 
\end{array}
\right).
\end{equation}

Now we state a local existence theorem which has been proved in  \cite{Kang_Lee_2013}. Although the result therein is stated at the fixed regularity $s=2$, the same argument extends in a straightforward manner to higher-order Sobolev spaces. We shall use the following standard local well-posedness in the regularity regime considered in this paper.    
\begin{theorem}\label{Local_Existence_Theorem}
Let $U_0\in H^s(\R^N ),\, s\geq [N/2]+2$ with $\nabla\cdot u_0=\nabla\cdot B_0=0$. Then, there exists $T=T(\Vert U_0\Vert_{H^s(\R^N )})$ small enough such that problem \eqref{Main_System} has a unique solution $U$ such that  
\begin{equation}
U\in L^\infty(0,T; H^s(\R^N ))\cap \mathrm{Lip} (0,T; H^{s-2}(\R^N )).
\end{equation}
\end{theorem}
\subsection{Global well-posedness }
In this subsection, we state  the global well-posedness of the solutions of \eqref{Main_System} and explain the main ideas of the proof.  Our approach is based on the energy method to construct an appropriate Lyapunov functional that captures the dissipative effect of all terms in the energy functional.     To control the nonlinear terms in the Maxwell system, we make use of the smoothing properties of the viscosity in the Navier-Stokes equation.

In order to state our main result, we introduce the time-weighted
energy norm and the corresponding dissipation norm as follows: For any $s\in\R^+$, we define 
$ \TE_s(t) $ as:

\begin{equation}
\TE_s^{2}(t)=\sum_{\ell=0}^{s}\sup_{0\leq \tau \leq t}\left( 1+\tau \right)
^{\ell}\Vert \nabla^{\ell}U\left( \tau \right) \Vert
_{H^{s-\ell}}^{2}, \label{Weighted_Energy}
\end{equation}
 and the corresponding dissipation norm, $\mathcal D_s(t)$
\begin{equation}\label{Dissipative_weighted_norm_1}
\begin{aligned}
\mathcal D_s^{2}(t) =\,&\sum_{\ell=0}^{s}\int_0^t\left((1+\tau)^\ell\Vert  \nabla^{\ell+1}u(\tau)\Vert_{H^{s-\ell}}^2+(1+\tau)^\ell \Vert \nabla^{\ell}E(\tau)\Vert_{H^{s-\ell}}^2\right)\dtau\\
&+\sum_{\ell=0}^{s-1}\int_0^t(1+\tau)^\ell\Vert\nabla^{\ell+1}B(\tau)\Vert_{H^{s-\ell-1}}^2 \dtau.
\end{aligned}
\end{equation} 
The choice of the  time weights in \eqref{Weighted_Energy} 
is dictated precisely by the expected decay. In fact, it 
encodes the gain of one factor $(1+t)^{-1/2}$ for each spatial derivative. Also, the specific form of the dissipation terms in \eqref{Dissipative_weighted_norm_1} reflects an important feature of the system \eqref{Main_System}: the three unknowns do not share the same dissipative structure. At the $k$-th level, the Navier-Stokes equation produces the usual parabolic dissipation $\mu\int_0^t \|\nabla^{k+1}u(\tau)\|_{L^2}^2\dtau$. This explains the presence of the first term in \eqref{Dissipative_weighted_norm_1}. The electric field $E$ is also directly damped through Ohm's law and produces a dissipation term of the form 
$\int_0^t \|\nabla^{k}E(\tau)\|_{L^2}^2\dtau$ which justifies the presence of the second term in \eqref{Dissipative_weighted_norm_1}.  For the magnetic field $B$, the situation is totally different; the third equation in \eqref{MHD_System} contains no dissipation, hence, the magnetic dissipation has to be recovered indirectly through the coupling between $E$ and $B$ in the Maxwell system. This explains why the dissipative term for $B$ is of the form \eqref{Dissipative_weighted_norm_1}  rather than a direct damping term involving $\|\nabla^\ell B\|_{H^{s-\ell}}$. 

Inspired by the decay properties in Lemma  \ref{Decay_Linear_Theor} and Lemma \ref{Lemma_E}, we define  the following quantities, which are convenient  in the computations below:
\begin{equation}\label{M_s}
\begin{aligned}
\mathcal M_s\left( t\right) =&\, \sum_{\ell=0}^{s}\sup_{0\leq \tau \leq
t}\left( 1+\tau \right) ^{N/4+\ell/2}\Vert \nabla^{\ell}U\left( \tau
\right) \Vert _{L^2}\\
&+\sum_{\ell=0}^{s-1}\sup_{0\leq \tau \leq
t}\left( 1+\tau \right) ^{N/4+1/2+\ell/2}\Vert \nabla^{\ell}E\left( \tau
\right) \Vert _{L^2}. 
\end{aligned}
\end{equation}
The upper bound in the second sum is natural since the damped equation of $E$ requires one additional derivative of $B$. 

To control the nonlinear terms arising in the weighted energy estimates, we introduce  the following time-weighted $L^\infty$ norms:
\begin{equation}\label{N_0}
{\mathcal N}_0(t)=\sup_{0\leq \tau\leq t}(1+\tau)^{N/2}\Vert  U(\tau)\Vert_{L^\infty}. 
\end{equation}
and 
\begin{equation}\label{N_1}
{\mathcal N}_1(t)=\sup_{0\leq \tau\leq t}(1+\tau)^{N/2+1/2}\Vert \nabla U(\tau)\Vert_{L^\infty}. 
\end{equation}
These quantities encode the expected decay of the nonlinear problem and arrive naturally when estimating differentiated nonlinearities. 
\begin{lemma}\label{Lemma_N_i_M_s}
For $s>\frac{N}{2}+1$, it holds that 
\begin{equation}
    {\mathcal N}_i(t)\lesssim \mathcal M_s\left( t\right),\qquad i=1,2. 
\end{equation}

\end{lemma}
\begin{proof}
Applying Lemma \ref{interpolation_lemma} with $\alpha =\frac{N}{2m},\,
q=r=2,\, j=0$ and $p=\infty$, we get for $m>\frac{N}{2}$
\begin{equation}
\left\Vert U\right\Vert _{L^{\infty}}\leq C\left\Vert \nabla^{m}%
U\right\Vert _{L^{2}}^{\frac{N}{2m}}\left\Vert U%
\right\Vert _{L^{2}}^{1-\frac{N}{2m}}. 
\end{equation}
This yields
\begin{eqnarray*}
{\mathcal N}_0(t)\lesssim  \mathcal M_s(t),
\end{eqnarray*}
provided that $s\geq m>\frac{N}{2}$.

Next, to estimate ${\mathcal N}_1(t)$, we apply Lemma \ref{interpolation_lemma} with $%
\alpha =\frac{N+2}{2 m},\, q=r=2,\, j=1$ and $p=\infty$, we get for $m>\frac{%
N}{2}+1$,
\begin{eqnarray*}
\left\Vert \nabla U\right\Vert _{L^{\infty}}\leq C\left\Vert
\nabla^{m}U\right\Vert _{L^{2}}^{\frac{N+2}{2 m}}\left\Vert U\right\Vert _{L^{2}}^{1-\frac{N+2}{2 m}}.
\end{eqnarray*}
This leads to
\begin{equation}
{\mathcal N}_1(t)\lesssim  \mathcal M_s(t),
\end{equation}
provided that $s\geq m> \frac{N}{2}+1$. This completed the proof of Lemma \ref{Lemma_N_i_M_s}. 
\end{proof}
Our main result in this paper reads as follows: 
\begin{theorem}
\label{Theorem_Global_existence}  Let $U_{0}\in H^{s}(\mathbb{R}^N)\cap L^{1}(\mathbb{R}^N)$
with $s\geq [N/2]+2$ and let  $\Lambda:=\left\Vert U_{0}\right\Vert
_{H^{s}}+\left\Vert U_{0}\right\Vert _{L^{1}}$. Then, there exists a
positive constant $\delta _{0}>0$ such that, if $\Lambda\leq \delta _{0}$, 
problem \eqref{MHD_System_Nonl} has a unique global solution $%
U$ satisfying%
\begin{equation}
U\in C\left( \left[ 0,\infty \right) ;H^{s}(\mathbb{R})\right) \cap C^{1}(%
\left[ 0,\infty \right) ;H^{s-1}(\mathbb{R})).  \label{regularity_class}
\end{equation}
Moreover, the solution satisfies the weighted energy estimate%
\begin{equation}
\mathcal{E}^{2}_s(t)+\mathcal{D}^{2}_s(t)\leq C\Lambda^2,  \label{weighted_estimate}
\end{equation}%
and the decay estimates 
\begin{subequations}\label{Decay_Main}
\begin{equation}
\Vert \nabla^{k}U\left( t\right)\Vert _{L^2}\leq
C\Lambda\left( 1+t\right) ^{-N/4-k/2}, \qquad 0\leq k\leq s
\label{Optimal_decay}
\end{equation}
and 
\begin{equation}
    \Vert \nabla^{k}E\left( t\right)\Vert _{L^2}\leq
C\Lambda\left( 1+t\right) ^{-N/4-1/2-k/2}, \qquad 0\leq k\leq s-1,
\label{Optimal_decay_E_Theorem}
\end{equation}
\end{subequations}
where $C$ is a positive constant.  
\end{theorem}
\subsubsection{Discussion of the main result}\label{Sec:Discussion}
 Before moving onto the proof, we briefly discuss the statements made above in Theorem \ref{Theorem_Global_existence}. 
 \begin{enumerate}
     \item[1.] The assumption $ U_{0}\in H^{s}(\mathbb{R}^N)\cap L^{1}(\mathbb{R}^N)$ can be relaxed and replaced by $U_{0}\in \dot{B}^{s,1}_2(\mathbb{R}^N)\cap \dot{B}^{-N/2}_{2,\infty}(\mathbb{R}^N)$ to get the same  decay rate for  the norm $\Vert \partial_kU\Vert_{\dot{B}^{s-k,1}_2(\mathbb{R}^N)}$.  The role of the $L^1$ assumption is essentially to provide control of the low frequencies in order to obtain the decay estimate. 
     \item[2.] It is clear from \eqref{Optimal_decay_E_Theorem} that the electric field enjoys a faster decay rate than the velocity and the magnetic fields. It gains an additional factor $(1+t)^{-1/2}$ compared to the general decay of $U$ in \eqref{Optimal_decay}, which is a consequence of the direct damping mechanism in the equation of $E$ and the Maxwell coupling. This improved rate plays an important role in the treatment of the nonlinear terms in the two-dimensional case, which allows us to overcome  the borderline time integrability of some nonlinear interactions and avoid a logarithmic loss in the decay rate. 
 \end{enumerate}

The proof of Theorem \ref{Theorem_Global_existence} will be given through several
steps. It is based on a combination of the time-weighted high-order energy estimates and time-decay estimates. More precisely, our objective is to establish control of the solution of \eqref{MHD_System_Nonl} uniformly in time  in the energy norm defined in \eqref{Weighted_Energy}.  Hence,   we aim  to prove the following estimates
\begin{subequations}  
\begin{equation}
\TE_s^{2}\left( t\right) +\mathcal D_s^{2}\left( t\right) \lesssim  \|U_0\|_{H^s}^2+\mathcal{E}_s(t)^2\mathcal{D}_s(t)^2+(\mathcal N_0(t)+\mathcal N_0^2(t)+\mathcal N_1(t))\mathcal D_s^2(t),  \label{Weighted_estimate}
\end{equation}
and 
\begin{equation}\label{Weighted_estimate_mass}
	\mathcal M_s(t)\lesssim (\|U_0\|_{L^1}+\|U_0\|_{H^s})+\mathcal M_s(t)\mathcal N_0(t)+\mathcal M_s^2(t)+\mathcal M_s^2(t)\mathcal N_0(t).
\end{equation}
\end{subequations}
The proof of \eqref{Weighted_estimate} will be given in Proposition \ref{Proposition_H_s_Estimate} and the one of \eqref{Weighted_estimate_mass} in Proposition \ref{proposition_Decay}. 



\section{Proof of the main results}\label{Section_Proof_Main_result}
This section is devoted to the proof of the main result of the paper. We begin in Subsection \ref{Subsection_Energy_Estimate} by establishing the weighted energy estimate and proving the estimate \eqref{Weighted_estimate}. In Subsection \ref{Subsection_Decay}, we prove \eqref{Weighted_estimate_mass} which is the main step in proving the decay estimates \eqref{Decay_Main}. Finally, we combine these two ingredients within a bootstrap argument to complete the proof of Theorem \ref{Theorem_Global_existence}. 
\subsection{Higher-order energy estimate--Proof of \eqref{Weighted_estimate}}\label{Subsection_Energy_Estimate}
The main step in the proof of 
\eqref{Weighted_estimate}  is to establish the time-weighted estimate separately at each derivative level. This is the content of the following proposition.  
\begin{proposition}\label{Proposition_H_s_Estimate}

It holds that 
\begin{subequations}\label{Estimate_induction}

\begin{equation}\label{Estimate_D_E}
\begin{aligned}
&\st\left( 1+\tau \right)
^{\ell}\Vert \nabla^{\ell}U\left( \tau \right) \Vert
_{H^{s-\ell}}^{2}\\
&+\int_0^t(1+\tau)^\ell\Big(\Vert  \nabla^{\ell+1}u(\tau)\Vert_{H^{s-\ell}}^2+\Vert \nabla^{\ell}E(\tau)\Vert_{H^{s-\ell}}^2    \Big)
\dtau\\
\lesssim&\, \|U_0\|_{H^s}^2+\mathcal{E}_s(t)^2\mathcal{D}_s(t)^2+\Big(\mathcal N_0(t)+\mathcal N_0^2(t)+\mathcal N_1(t)\Big)\mathcal D_s^2(t),\quad 0\leq \ell\leq s,
\end{aligned}
\end{equation}
and 
\begin{equation}\label{Estimate_D_E_2}
    \begin{aligned}
&\int_0^t(1+\tau)^\ell         \Vert \nabla^{\ell+1}B\Vert_{H^{s-\ell-1}}^2       \;
\dtau\\
\lesssim&\,\|U_0\|_{H^s}^2+\mathcal{E}_s(t)^2\mathcal{D}_s(t)^2+\Big(\mathcal N_0(t)+\mathcal N_0^2(t)+\mathcal N_1(t)\Big)\mathcal D_s^2(t),\quad 0\leq \ell\leq s-1.
\end{aligned}
\end{equation}
\end{subequations}
\end{proposition}
The estimate \eqref{Estimate_D_E} below controls directly the component $u$ and $E$, whereas the estimate \eqref{Estimate_D_E_2} recovers the missing dissipation of the magnetic field. Summing these estimates over the different derivative levels yields \eqref{Weighted_estimate}.

Now, we are in a position to state the key proposition that will enable us to prove \eqref{Estimate_D_E} and \eqref{Estimate_D_E_2}.

\begin{proposition}
    \label{Estimate_E_D_lemma}
	Assuming that $ U=(u,E,B) $ is a regular solution to \eqref{MHD_System_Nonl}. Then we have the following a priori estimates.
 \begin{subequations} \begin{equation}\label{Estimate_U0_lemma}
  \begin{aligned}
      &\st \left( 1+\tau\right) ^{\alpha }\left\Vert U\left( \tau\right) \right\Vert
					_{L^{2}}^{2}+2\int_{0}^{t}\left( 1+\tau\right) ^{\alpha}\left(\mu\Vert \nabla u(\tau)\Vert_{L^2}^2+\Vert E(\tau)\Vert_{L^2}^2\right)\dtau  \\
				\leq &\left\Vert U_{0}\right\Vert _{L^{2}}^{2}+\alpha \int_{0}^{t}\left(
				1+\tau\right) ^{\alpha -1}\left\Vert U\left( \tau\right) \right\Vert _{L^{2}}^{2}\dtau\\
				&+C\mathcal{E}_s(t)^2\int_0^t (1+\tau)^\alpha(\|\nabla u(\tau)\|_{L^2}^2+\|\nabla B(\tau)\|_{L^2}^2)\dtau .
    \end{aligned}
		\end{equation}
		Also,  for $ 1\le k \le s $, we have 
\begin{equation}\label{Estimate_Uk_lemma}
  \begin{aligned}
		&\st\left( 1+\tau\right) ^{\alpha }\Vert \nabla^{k} U\left( \tau\right) \Vert
			_{L^{2}}^{2}+2\int_{0}^{t}\left( 1+\tau\right) ^{\alpha}\left(\mu\Vert \nabla^{k+1} u(\tau)\Vert_{L^2}^2+\Vert \nabla^{k} E(\tau)\Vert_{L^2}^2\right)\dtau   \\
		\leq\, &\Vert \nabla^{k}U_{0}\Vert _{L^{2}}^{2}+\alpha \int_{0}^{t}\left(
		1+\tau\right) ^{\alpha -1} \Vert \nabla^{k}U\left( \tau\right)\Vert _{L^{2}}^{2}\dtau
		+C\mathcal N_0(t)\int_{0}^{t}(1+\tau)^{-N/2+\alpha} \Vert \nabla^{k}{U}(\tau)\Vert_{L^2}^2\dtau\\
		&+C\mathcal N_0^2(t)\int_{0}^{t}(1+\tau)^{-N+\alpha} \Vert \nabla^{k}{U}(\tau)\Vert_{L^2}^2\dtau\\
		&+C\mathcal N_1(t)\int_{0}^{t}(1+\tau)^{-N/2-1/2+\alpha} \Vert \nabla^{k}{U}(\tau)\Vert_{L^2}^2\dtau.
  \end{aligned}
		\end{equation}
		In addition, for $ 0\le k \le s-1 $, it holds that 
\begin{equation}\label{Estimate_time_integral_B_lemma}
  \begin{aligned}
			&\int_0^t(1+t)^{\alpha}\Vert \nabla^{k+1} B\Vert_{L^2}^2\dtau\\
			\le&\,C\Vert \nabla^{k} U_0\Vert_{H^1}^2
				+C\st(1+\tau)^{\alpha}\Vert \nabla^{k} U\Vert_{H^1}^2
				+ C \int_{0}^{t}(1+\tau)^{\alpha} (\Vert \nabla^{k} E\Vert_{L^2}^2+\Vert \nabla^{k+1}E \Vert _{L^2}^2)\; \dtau\\
			&+C\mathbf{1}_{k\geq 1}\mathcal N_0^2(t)\int_{0}^{t}(1+\tau)^{-N+{\alpha}}\Vert
				\nabla^{k} U(\tau)\Vert _{L^2}^2\;\dtau+\alpha\int_0^t(1+\tau)^{\alpha-1}\|\nabla^{k+1}B\|^2_{L^2}\; \dtau\\
				&+C\int_0^t(1+\tau)^\alpha (\|u(\tau)\|_{L^2}^2+\|B(\tau)\|_{L^2}^2)(\|\nabla u(\tau)\|_{L^2}^2+\|\nabla B(\tau)\|_{L^2}^2)\dtau,
    \end{aligned}
		\end{equation}
	
where $ \alpha \in \R$. 
\end{subequations}
\end{proposition}

\begin{proof}
\textbf{Proof of \eqref{Estimate_U0_lemma}.}
Multiplying the first
equation in (\ref{MHD_System}) 
by $u$, the second equation by $E$ and the third
equation by $B$, 
adding the resulting equations,  integrating with respect to $x$ over $%
\mathbb{R}^N$ and using integration by parts, 
we get 
\begin{eqnarray}\label{dE_dt_0_1}
\ddt \left(\Vert u\Vert_{L^2}^2+\Vert B\Vert_{L^2}^2+\Vert E\Vert_{L^2}^2\right)+2\mu\Vert \nabla u\Vert_{L^2}^2+2\Vert j\Vert_{L^2}^2=0.
\end{eqnarray}
 Now, multiplying  the last equation in \eqref{MHD_System} by $j$ and integrating over $\R^N $, we have 
\begin{equation}
\Vert j\Vert_{L^2}^2=\Vert E\Vert_{L^2}^2+\Vert u\times B\Vert_{L^2}^2+2\int_{\R^N } (u\times B)\cdot E\dx. 
\end{equation}
That is 
\begin{equation}
\Vert j\Vert_{L^2}^2\geq \Vert E\Vert_{L^2}^2+2\int_{\R^N } (u\times B)\cdot E\dx. 
\end{equation}
Consequently, \eqref{dE_dt_0_1} takes the form 
\begin{eqnarray}\label{dE_dt_0_2}
\begin{aligned}
&\ddt \left(\Vert u\Vert_{L^2}^2+\Vert B\Vert_{L^2}^2+\Vert E\Vert_{L^2}^2\right)+2\mu\Vert \nabla u\Vert_{L^2}^2+2\Vert E\Vert_{L^2}^2\\
\leq&\,-4\int_{\R^N } (u\times B)\cdot E\dx.
\end{aligned}
\end{eqnarray}
Applying H\"older's inequality together with Young's inequality, we obtain for any $\epsilon_0>0$,
\begin{eqnarray}\label{dE_dt_0_3}
\begin{aligned}
&\ddt \left(\Vert u\Vert_{L^2}^2+\Vert B\Vert_{L^2}^2+\Vert E\Vert_{L^2}^2\right)+2\mu\Vert \nabla u\Vert_{L^2}^2+2(1-\epsilon_0)\Vert E\Vert_{L^2}^2\\
 \leq&\, c(\epsilon_0)\Vert u\times B\Vert _{L^2}^2 .
\end{aligned}
\end{eqnarray}
Now, applying H\"older's inequality together with the 
Ladyzhenskaya's inequality:
\begin{equation}\label{Lady_Inequality}
\begin{aligned}
    \| f\|_{L^4(\R^N }
    \leq  \, C \|\nabla f\|_{L^2(\R^N )}^{N/4}\|f\|_{L^2(\R^N )}^{1-N/4}, \qquad 1\leq N\leq 4,
        \end{aligned}
	\end{equation}
    we obtain 
\begin{equation}\label{Ineq_u_B_0}
\begin{aligned}
\Vert u\times B(t)\Vert _{L^2}^2\lesssim&\, \|u\|_{L^4}^2\|B\|_{L^4}^2\\
\lesssim &\,\|\nabla u\|_{L^2}^{N/2}\|u\|_{L^2}^{2(1-N/4)}\|\nabla B\|_{L^2}^{N/2}\|B\|_{L^2}^{2(1-N/4)}
\end{aligned}
\end{equation}
Using Young's inequality in the product involving $u$ and $B$, we obtain 
\begin{equation}\label{Estim_Prod}
    \begin{aligned}
        \Vert u\times B(t)\Vert _{L^2}^2\lesssim&\,(\|u\|_{L^2}^2+\|B\|_{L^2}^2)^{2-N/2}(\|\nabla u\|_{L^2}^2+\|\nabla B\|_{L^2}^2)^{N/2}. 
    \end{aligned}
\end{equation}
On the other hand, using the definition of $\mathcal{E}_s(t)$ and since $s\geq 1$, we have 
\begin{equation}\label{H_1_Est}
\|u\|_{H^1}^2+\|B\|_{H^1}^2\lesssim \mathcal{E}_s(t)^2. 
\end{equation}
Consequently, since $N=2,\, 3$, we obtain for \eqref{Estim_Prod} and \eqref{H_1_Est}
\begin{equation}\label{Ineq_u_B_0}
\begin{aligned}
    \Vert u\times B(t)\Vert _{L^2}^2\lesssim&\,\mathcal{E}_s(t)^{4-N}\mathcal{E}_s(t)^{N-2}(\|\nabla u\|_{L^2}^2+\|\nabla B\|_{L^2}^2)\\
    \lesssim&\,\mathcal{E}_s(t)^2(\|\nabla u\|_{L^2}^2+\|\nabla B\|_{L^2}^2).
    \end{aligned}
\end{equation}
Plugging \eqref{Ineq_u_B_0} into \eqref{dE_dt_0_3}, choosing $\epsilon_0<1$,  then multiplying the result  by $(1+t)^\alpha$, integrating with respect to  $t$ over $(0,t)$ and using \eqref{M_s}, we obtain 
\begin{equation}
    \begin{aligned}
&\st \left( 1+\tau\right) ^{\alpha }\left\Vert U\left( \tau\right) \right\Vert
_{L^{2}}^{2}+2\int_{0}^{t}\left( 1+\tau\right) ^{\alpha}\left(\Vert \nabla u(\tau)\Vert_{L^2}^2+\Vert E(\tau)\Vert_{L^2}^2\right)\dtau  \notag \\
\leq &\left\Vert U_{0}\right\Vert _{L^{2}}^{2}+\alpha \int_{0}^{t}\left(
1+\tau\right) ^{\alpha -1}\left\Vert U\left( \tau\right) \right\Vert _{L^{2}}^{2}\dtau\notag\\
&+C\mathcal{E}_s(t)^2\int_0^t (1+\tau)^\alpha(\|\nabla u(\tau)\|_{L^2}^2+\|\nabla B(\tau)\|_{L^2}^2)\dtau.
\label{First_estimate_U}
\end{aligned}
\end{equation}
This finishes the proof of \eqref{Estimate_U0_lemma}.  

\textbf{Proof of \eqref{Estimate_Uk_lemma}.}
We now derive  the energy estimates for higher-order spatial  derivatives. Let $1\leq k\leq s$. Applying $\nabla^k$ to (\ref{MHD_System}) and introduce the notation  
\begin{equation}
    \tilde{g}=\nabla^{k} g,\qquad \text{for}\quad g\in\{u,B,E,p,j,U\},
\end{equation} 
we obtain 
\begin{eqnarray}\label{System_k}
\left\{
\begin{array}{ll}
\partial_t \tilde{u}-\mu\Delta \tilde{u}+\nabla \tilde{p}+(u\cdot \nabla)\tilde{u}=\nabla^{k}(j\times B)+f,\\[4pt]
\partial_t \tilde{E}-\nabla\times \tilde{B}=-\tilde{j},\\[4pt]
\partial_t \tilde{B}+\nabla\times \tilde{E}=0,\\[4pt]
\nabla\cdot \tilde{u}=\nabla\cdot \tilde{B}=0,\\[4pt]
\tilde{E}+\nabla^{k}(u\times B)=\tilde{j},
\end{array}
\right. 
\end{eqnarray}
where 
$f=-[\nabla^{k}, (u\cdot \nabla)]u=(u\cdot\nabla)\nabla^k u-\nabla^k((u\cdot\nabla)u)$.   

Now, multiplying the first equation in system \eqref{System_k} by $\tilde{u}$, the second equation by $\tilde{E}$ and the third equation by $\tilde{B}$, integrating over $\R^N $ and adding the resulting equations, we obtain as before 
  \begin{equation}
  \begin{aligned}
       \label{Energy_Indentity_k}
&\frac12\ddt \mathcal{E}^k(t)+\mu\Vert \nabla \tilde{u}\Vert_{L^2}^2+\Vert \tilde{E}\Vert_{L^2}^2\\
=&\int_{\R^N } \nabla^{k}(j\times B)\cdot\tilde{u}\dx-\int_{\R^N }\nabla^{k}(u\times B)\cdot\tilde{E}\dx+\int_{\R^N } f\cdot \tilde{u}\dx, 
\end{aligned}
\end{equation}
where 
\begin{equation}\label{Energy_Equation_k_tilde}
\mathcal{E}^k(t):=\Vert \tilde{u}\Vert_{L^2}^2+\Vert \tilde{E}\Vert_{L^2}^2 +\Vert \tilde{B}\Vert_{L^2}^2=\Vert \nabla^k U\Vert_{L^2}^2. 
\end{equation}
Using the last equation in the system \eqref{MHD_System}, we have
\begin{equation}
\int_{\R^N } \nabla^{k}(j\times B)\cdot\tilde{u}\dx=\int_{\R^N } \nabla^{k}(E\times B)\cdot\tilde{u}\dx+\int_{\R^N } \nabla^{k}((u\times B)\times B)\cdot\tilde{u}\dx.
\end{equation}
Now, thanks to H\"older's inequality, we find 
\begin{equation}\label{Est_First_term}
\left\vert \int_{\R^N }\nabla^{k}(E\times B)\cdot\tilde{u}\dx\right\vert\leq \Vert \nabla^{k}(E\times B) \Vert_{L^2}\Vert \tilde{u}\Vert_{L^2}. 
\end{equation}
Applying \eqref{First_inequaliy_Guass} for $p=2,\, r=2$ and $q=\infty$ and recalling \eqref{N_0}, we obtain the following estimate  
\begin{equation}
    \label{B_estimate_1}
    \begin{aligned}
\Vert \nabla^{k}(E\times B) \Vert_{L^2}\lesssim \,&\Vert E\Vert _{L^\infty}\Vert
\nabla^{k}B\Vert _{L^2}+\Vert B\Vert _{L^\infty}\Vert \nabla^{k}E\Vert
_{L^2}\notag\\
\lesssim\, &\mathcal N_0(t)(1+t)^{-N/2} \Vert \tilde{U}\Vert_{L^2}. 
\end{aligned}
\end{equation}
Hence, plugging this last estimate into \eqref{Est_First_term}, we obtain 
\begin{equation}\label{Est_First_Term}
\left\vert \int_{\R^N }\nabla^{k}(E\times B)\cdot\tilde{u}\dx\right\vert\leq C\mathcal N_0(t)(1+t)^{-N/2} \Vert \tilde{U}\Vert_{L^2}^2.
\end{equation}
By the same method, we have the estimate 
\begin{equation}  
\label{Mid_Term}
\left\vert\int_{\R^N }\nabla^{k}(u\times B)\cdot\tilde{E}\dx\right\vert\leq C\mathcal N_0(t)(1+t)^{-N/2} \Vert \tilde{U}\Vert_{L^2}^2.
\end{equation}
Moreover, 
\begin{equation}
    \begin{aligned}
\left\vert\int_{\R^N } \nabla^{k}((u\times B)\times B)\cdot\tilde{u}\dx\right\vert\leq\,& \Vert
\nabla^{k}((u\times B)\times B)\Vert _{L^2} \Vert
\tilde{u}\Vert _{L^2}\\
\lesssim&\, \Big( \Vert B\Vert _{L^\infty}\Vert
\nabla^{k}(u\times B)\Vert _{L^2}+\Vert u\times B\Vert _{L^\infty}\Vert \nabla^{k}B\Vert
_{L^2}\Big)\Vert
\tilde{U}\Vert _{L^2}. 
\end{aligned}
\end{equation}
Applying \eqref{B_estimate_1}, we obtain 
\begin{equation}\label{Estimate_Second_Term_1}
\left\vert\int_{\R^N } \nabla^{k}((u\times B)\times B)\cdot\tilde{u}\dx\right\vert\lesssim   \mathcal N_0^2(t)(1+t)^{-N} \Vert
\tilde{U}\Vert _{L^2}^2.
\end{equation}
Hence, collecting  \eqref{Est_First_Term}, \eqref{Mid_Term} and \eqref{Estimate_Second_Term_1},  we obtain  that 
\begin{equation}
    \label{Estimate_First_Term}
\left\vert \int_{\R^N } \nabla^{k}(j\times B)\cdot\tilde{u}\dx\right\vert
\lesssim \mathcal N_0(t)(1+t)^{-N/2} \Vert \tilde{U}\Vert_{L^2}^2 + \mathcal N_0^2(t)(1+t)^{-N} \Vert
\tilde{U}\Vert _{L^2}^2.
\end{equation}
We next estimate the commutator term appearing on the left-hand side of \eqref{Energy_Indentity_k}. 
Recalling $f=-[\nabla^{k}, (u\cdot \nabla)]u$ and $\tilde{u}=\nabla^k u$, we have 
\begin{equation}\label{Equa_tilde_u}
\begin{aligned}
\int_{\R^N } f \tilde{u}\dx=&\,-\int_{\R^N }[\nabla^{k}, (u\cdot \nabla)]u\cdot\tilde{u}\dx\\
=&\,\int_{\R^N }\left(\sum_{i=1}^N[\nabla^{k}, (u_i\partial_i)]\right)u\cdot\tilde{u}\dx\notag\\
=&\,\int_{\R^N }\left(\sum_{i=1}^N\nabla^{k}(u^i\partial_i u)-u^i\partial_i\nabla^{k}u\right)\cdot\tilde{u}\dx\notag\\
=&\,\int_{\R^N }\left(\sum_{i=1}^N\nabla^{k}(u^i\partial_i u)-u^i\nabla^{k}\partial_iu\right)\cdot\tilde{u}\dx. 
\end{aligned}
\end{equation}
Writing the term $\nabla^{k}(u^i\partial_i u)-u^i\nabla^{k}\partial_iu=[\nabla^{k},u^i]\partial_iu$, and applying the commutator estimate \eqref{Second_inequality_Gauss} and recalling \eqref{N_1}, we obtain
\begin{equation}\label{Estimate_Second_Term}
\begin{aligned}
\left\vert\int_{\R^N } f \tilde{u}dx\right\vert\lesssim\,& \,\Vert  \nabla  U\Vert_{L^\infty}
\Vert \tilde{U}\Vert_{L^2}^2\\
\lesssim\,&\,\mathcal N_1(t)(1+t)^{-N/2-1/2}\Vert \tilde{U}\Vert_{L^2}^2.
\end{aligned}
\end{equation}
Now plugging the estimates \eqref{Mid_Term}, \eqref{Estimate_First_Term}  and \eqref{Estimate_Second_Term} into \eqref{Energy_Indentity_k}, we obtain for $k\geq 1$
\begin{equation}\label{Energy_Indentity_k_2}
\begin{aligned}
&\frac12\ddt \mathcal{E}^k(t)+\mu\Vert \nabla \tilde{u}\Vert_{L^2}^2+\Vert \tilde{E}\Vert_{L^2}^2\\
\lesssim &\,\mathcal N_0(t)(1+t)^{-N/2} \Vert \tilde{U}\Vert_{L^2}^2+\mathcal N_0^2(t)(1+t)^{-N} \Vert
\tilde{U}\Vert _{L^2}^2\\
&+ N_1(t)(1+t)^{-N/2-1/2}\Vert \tilde{U}\Vert_{L^2}^2.
\end{aligned}
\end{equation}
To derive  time-weighted higher order energy estimate, we multiply  \eqref{Energy_Indentity_k_2} by $(1+t)^\alpha$ , then using  \eqref{Energy_Equation_k_tilde} and  integrating the resulting inequality  from $0$ to $t$, we  obtain 
\begin{equation}
\begin{aligned}
\label{Estimate_U_k}
&\st\left( 1+\tau\right) ^{\alpha }\Vert \nabla^{k} U\left( \tau\right) \Vert
_{L^{2}}^{2}+2\int_{0}^{t}\left( 1+\tau\right) ^{\alpha}\left(\mu\Vert \nabla^{k+1} u(\tau)\Vert_{L^2}^2+\Vert \nabla^{k} E(\tau)\Vert_{L^2}^2\right)\dtau   \\
\lesssim \,&\Vert \nabla^{k}U_{0}\Vert _{L^{2}}^{2}+\alpha \int_{0}^{t}\left(
1+\tau\right) ^{\alpha -1} \Vert \nabla^{k}U\left( \tau\right) \Vert _{L^{2}}^{2}\dtau
+\mathcal N_0(t)\int_{0}^{t}(1+\tau)^{-N/2+\alpha} \Vert \nabla^{k}{U}(\tau)\Vert_{L^2}^2\dtau\\
&+\mathcal N_0^2(t)\int_{0}^{t}(1+\tau)^{-N+\alpha} \Vert \nabla^{k}{U}(\tau)\Vert_{L^2}^2\dtau
+\mathcal N_1(t)\int_{0}^{t}(1+\tau)^{-N/2-1/2+\alpha} \Vert \nabla^{k}{U}(\tau)\Vert_{L^2}^2\dtau.
\end{aligned}
\end{equation}
This finishes the proof of \eqref{Estimate_Uk_lemma}

\textbf{Proof of \eqref{Estimate_time_integral_B_lemma}} Now, taking the curl of the second equation in  system \eqref{System_k}, we obtain
\begin{equation}\label{eq_curl_1_k}
\partial_t (\nabla\times\tilde{E})-\nabla\times(\nabla\times \tilde{B})=-\nabla\times\tilde{j}.
\end{equation}
 Since $\nabla\cdot \tilde{B}=0$, then, we have the identity 
 \begin{equation}
\nabla\times \nabla \times \tilde{B} =-\Delta \tilde{B}. 
\end{equation}
Therefore equation \eqref{eq_curl_1_k} becomes 
 \begin{equation}
\partial_t (\nabla\times\tilde{E})+\Delta \tilde{B}=-\nabla\times\tilde{j}.
\end{equation}
Now, multiplying the above equation by $-\tilde{B}$ and integrating over $\R^N ,$ we get 
\begin{equation}\label{ddt_E}
-\int_{\R^N } \partial_t (\nabla\times\tilde{E})\cdot \tilde{B}\dx +\Vert \nabla \tilde{B}\Vert_{L^2}=\int_{\R^N }(\nabla\times\tilde{j}) \cdot\tilde{B}\; \dx. 
\end{equation}
Using the third equation in  \eqref{System_k}, we obtain   
\begin{equation}\label{ddt_B}
-\int_{\R^N }(\nabla\times\tilde{E})\cdot \partial_t \tilde{B}\dx=\int_{\R^N } |\nabla\times\tilde{E}|^2\dx.
\end{equation}
Collecting \eqref{ddt_E} and \eqref{ddt_B}, we find  
\begin{equation}\label{Eq_2_1_k}
\ddt F_1^k(t)
+\Vert \nabla \tilde{B}\Vert_{L^2}=-\int_{\R^N }(\nabla\times\tilde{j})\cdot \tilde{B} \dx+\int_{\R^N }|\nabla\times\tilde{E}|^2\dx, 
\end{equation} 
where 
\begin{equation}
F_1^k(t)=-\int_{\R^N } (\nabla\times \nabla^{k} E)\cdot \nabla^{k} B\dx .
\end{equation}
Now using the identity 
\begin{equation}
(\nabla \times\tilde{j})\cdot \tilde {B}=\nabla\cdot (\tilde{j}\times \tilde {B})+\tilde{j}\cdot(\nabla \times\tilde {B})
\end{equation}
then the first term on the right-hand side  of \eqref{Eq_2_1_k} can be estimated as
\begin{equation}
    \begin{aligned} 
\left\vert\int_{\R^N }(\nabla\times\tilde{j})\cdot \tilde{B}\dx \right \vert\,&= \int_{\R^N }|\tilde{j}\cdot(\nabla \times\tilde {B})|\dx
\leq \,
 \Vert \nabla^{k} j\Vert_{L^2}\Vert \nabla^{k+1} B\Vert_{L^2.
 }.
\end{aligned}
\end{equation}
Hence, applying Young's inequality, we get for any $\varepsilon>0$, 
 \begin{eqnarray*}
\left\vert\int_{\R^N }(\nabla\times\tilde{j}) \tilde{B} dx \right \vert\leq  c(\varepsilon)\Vert \nabla^{k} j\Vert_{L^2}^2+\varepsilon \Vert \nabla^{k+1} B\Vert_{L^2}^2. 
\end{eqnarray*}
We also have the estimate 
\begin{equation}
\int_{\R^N }|\nabla\times\tilde{E}|^2dx\leq  \Vert \nabla^{k+1}E \Vert _{L^2}^2.
\end{equation}
On the other hand, we have by the last equation in \eqref{System_k}, 
\begin{equation}
\nabla^{k}E+\nabla^{k}(u\times B)=\nabla^{k}j.
\end{equation}
This implies that for $k\geq 0$,

\begin{equation}
 \Vert \nabla^{k}j \Vert{ _{L^2}^2 }\leq  2(\Vert \nabla^{k}E \Vert{ _{L^2}^2 }+\Vert \nabla^{k}(u\times B) \Vert{ _{L^2}^2 }).
\end{equation}
Now, similar to the estimate \eqref{B_estimate_1}, we have for $1\leq k\leq s-1$
\begin{equation}
\Vert \nabla^{k}(u\times B) \Vert{ _{L^2}^2 }\lesssim \mathcal N_0^2(t)(1+t)^{-N} \Vert \nabla^{k} U\Vert_{L^2}^2,\quad k\geq 0, 
\end{equation}
and for $k=0$,  we have the estimate \eqref{Ineq_u_B_0}. 
Hence, for $k\geq 1$, we have 
\begin{equation}
\Vert \nabla^{k}j \Vert{ _{L^2}^2 }\leq  2\Vert \nabla^{k}E \Vert{ _{L^2}^2 }+C\mathcal N_0^2(t)(1+t)^{-N} \Vert \nabla^{k} U\Vert_{L^2}^2
\end{equation}
and for $k=0$, we have 
\begin{equation}
\Vert j \Vert{ _{L^2}^2 }\leq  2\Vert E \Vert{ _{L^2}^2 }+(\|u\|_{L^2}^2+\|B\|_{L^2}^2)(\|\nabla u\|_{L^2}^2+\|\nabla B\|_{L^2}^2).
\end{equation}
Plugging the above estimates into \eqref{Eq_2_1_k}, we get  for
 $k\geq 0$, 
\begin{equation}\label{Eq_2_2_k}
\begin{aligned}
&\ddt F^k(t)
+(1-\varepsilon)\Vert \nabla^{k+1} B\Vert_{L^2}^2\\
\leq\,& c(\varepsilon)(\Vert \nabla^{k} E\Vert_{L^2}^2+\Vert \nabla^{k+1}E \Vert _{L^2}^2) 
+C\mathbf{1}_{k\geq 1}\mathcal N_0^2(t)(1+t)^{-N} \Vert \nabla^{k} U\Vert_{L^2}^2\\
&+(\|u\|_{L^2}^2+\|B\|_{L^2}^2)(\|\nabla u\|_{L^2}^2+\|\nabla B\|_{L^2}^2).
\end{aligned}
\end{equation}
The estimate \eqref{Eq_2_2_k} shows that although the equation of $B$ itself does not provide any direct damping, the interaction between $E$ and $B$ generate through the cross term $F^k(t)$ the coercive quantity $\|\nabla^{k+1}B\|_{L^2}^2$ which in turn explain the compensated structure of the Maxwell system. 
 
Now, it is not hard to see that  
\begin{equation}\label{Equv_L_U}
| F^k(t)|\leq c_2 \Vert
\nabla^{k} U\Vert _{H^1}^2. 
\end{equation}
Consequently, we fix $\varepsilon<1$  and  multiply \eqref{Eq_2_2_k} by $(1+t)^{\alpha}$, integrating from $0$ to $t$ and taking into account \eqref{Equv_L_U}, we get for $0\leq k\leq s-1$, 
\begin{equation}
    \begin{aligned}
\label{Estimate_main_L_2}
&\int_0^t(1+\tau)^{\alpha}\Vert \nabla^{k+1} B\Vert_{L^2}^2 \textup{d} \tau\\
\lesssim&\,\Vert \nabla^{k} U_0\Vert_{H^1}^2
+\st(1+\tau)^{\alpha}\Vert \nabla^{k} U\Vert_{H^1}^2
+ \int_{0}^{t}(1+\tau)^{\alpha} (\Vert \nabla^{k} E\Vert_{L^2}^2+\Vert \nabla^{k+1}E \Vert _{L^2}^2) \dtau\\
&+\mathbf{1}_{k\geq 1}\mathcal N_0^2(t)\int_{0}^{t}(1+\tau)^{-N+{\alpha}}\Vert
\nabla^{k} U(\tau)\Vert _{L^2}^2\dtau+\alpha\int_0^t(1+\tau)^{\alpha-1}\|\nabla^{k+1} B\|^2_{L^2}\;\dtau\\
&+\int_0^t(1+\tau)^\alpha (\|u(\tau)\|_{L^2}^2+\|B(\tau)\|_{L^2}^2)(\|\nabla u(\tau)\|_{L^2}^2+\|\nabla B(\tau)\|_{L^2}^2). 
 \end{aligned}
\end{equation}
This yields \eqref{Estimate_time_integral_B_lemma} and ends the proof of Proposition \ref{Estimate_E_D_lemma}. 
\end{proof}

\subsubsection{Proof of Proposition \ref{Proposition_H_s_Estimate}}
We now complete the proof of \eqref{Estimate_induction} using an induction argument in $\ell$.
We begin with the base case $ \ell = 0 $. 
Setting  $ \alpha=0 $ in Proposition \ref{Estimate_E_D_lemma}, summing \eqref{Estimate_Uk_lemma} over  $ 1\leq k\leq s $, combining the resulting estimate  with  \eqref{Estimate_U0_lemma} and  using the estimate 
\begin{equation}
\mathcal{E}_s(t)^2\int_0^t (\|\nabla u(\tau)\|_{L^2}^2+\|\nabla B(\tau)\|_{L^2}^2)\dtau\lesssim  \mathcal{E}_s(t)^2\mathcal{D}_s(t)^2,
\end{equation}
 we  obtain
\begin{equation}\label{induction_l=0_u_E}
\begin{aligned}
	&\st \left\Vert U\left( \tau\right) \right\Vert_{H^s}^{2}+2\int_{0}^{t}\left(\mu\Vert \nabla  u(\tau)\Vert_{H^s}^2+\Vert E(\tau)\Vert_{H^s}^2\right)\dtau\\
\lesssim\,& \left\Vert U_{0}\right\Vert _{H^s}^{2}+\mathcal{E}_s(t)^2\mathcal{D}_s(t)^2+\mathcal{N}_0(t)\int_{0}^{t}(1+\tau)^{-N/2} \Vert \nabla {U}(\tau)\Vert_{H^{s-1}}^2\dtau
\\
&+\mathcal{N}_0^2(t)\int_{0}^{t}(1+\tau)^{-N} \Vert \nabla {U}(\tau)\Vert_{H^{s-1}}^2\dtau\\
&+\mathcal{N}_1(t)\int_{0}^{t}(1+\tau)^{-N/2-1/2} \Vert \nabla{U}(\tau)\Vert_{H^{s-1}}^2\dtau.
\end{aligned}
\end{equation}

Summing \eqref{Estimate_time_integral_B_lemma} for $ 0\le k \le s-1 $
, one has
\begin{equation}\label{induction_l=0_B}
\begin{aligned}
	\int_0^t
	\Vert \nabla B\Vert_{H^{s-1}}^2\dtau
	\lesssim\,&\Vert U_0\Vert_{H^s}^2
		+\st\Vert U(\tau)\Vert_{H^s}^2+C\mathcal{E}_s(t)^2\mathcal{D}_s(t)^2\\
	&+  \int_{0}^{t}
	\|E\|_{H^s} \dtau+\mathcal N_0^2(t)\int_{0}^{t}(1+\tau)^{-N}\Vert
		\nabla  U(\tau)\Vert _{H^{s-1}}^2\dtau\\
	\lesssim\,&\Vert U_0\Vert_{H^s}^2
	+\st\Vert U\Vert_{H^s}^2+\mathcal{E}_s(t)^2\mathcal{D}_s(t)^2\\
 &+ \int_{0}^{t} \|E\|_{H^s} \dtau+\mathcal N_0^2(t)\int_{0}^{t}(1+\tau)^{-N}\Vert
	\nabla U(\tau)\Vert _{H^{s}}^2\dtau.
 \end{aligned}
\end{equation}
Now, adding \eqref{induction_l=0_B} + $ \lambda $ \eqref{induction_l=0_u_E}, where  $ \lambda>0 $ is chosen sufficiently small, yields 
\begin{equation}\label{induction_l=0}
\begin{aligned}
	&\st \left\Vert U\left( \tau\right) \right\Vert_{H^s}^{2}+\int_{0}^{t}\left(\Vert \nabla  u(\tau)\Vert_{H^s}^2+\Vert E(\tau)\Vert_{H^s}^2+\|\nabla B(\tau\|_{H^{s-1}}\right)\dtau\\
	\lesssim\,&\left\Vert U_{0}\right\Vert _{H^s}^{2}+\mathcal{E}_s(t)^2\mathcal{D}_s(t)^2+\mathcal N_0(t)\int_{0}^{t}(1+\tau)^{-N/2} \Vert \nabla {U}(\tau)\Vert_{H^{s-1}}^2\dtau\\
	&+\mathcal N_0^2(t)\int_{0}^{t}(1+\tau)^{-N} \Vert \nabla {U}(\tau)\Vert_{H^{s-1}}^2\dtau\\
	&+\mathcal N_1(t)\int_{0}^{t}(1+\tau)^{-N/2-1/2} \Vert\nabla U(\tau)\Vert_{H^{s-1}}^2\dtau\\
	\lesssim &\, \left\Vert U_{0}\right\Vert _{H^s}^{2}+\mathcal{E}_s(t)^2\mathcal{D}_s(t)^2+\rm{R}_1^{0}+ \rm{R}_2^{0}+\rm{R}_3^{0},
 \end{aligned}
\end{equation}
where $\rm{R}_1^{0},\, \rm{R}_2^{0}$ and $\rm{R}_3^{0}$
stand for the three time-integral terms on the right-hand side of the above estimate in the order in which they appear. 

\noindent
We have 
\begin{equation}
\begin{aligned}
\rm{R}_1^{0}+\rm{R}_2^{0}+\rm{R}_3^{0}\lesssim\,& C (\mathcal{N}_0(t)+\mathcal{N}_1(t)+\mathcal{N}_0^2(t))\int_{0}^{t} \Vert \nabla{U}(\tau)\Vert_{H^{s-1}}^2\dtau\\
\lesssim\,& (\mathcal{N}_0(t)+\mathcal{N}_1(t)+\mathcal{N}_0^2(t))\mathcal D_s^{2}(t). 
\end{aligned}
\end{equation}
Hence, plugging the above estimate into \eqref{induction_l=0}, we obtain 
\begin{equation}
\begin{aligned}
&\st \left\Vert U\left( \tau\right) \right\Vert_{H^s}^{2}+\int_{0}^{t}\left(\Vert \nabla  u(\tau)\Vert_{H^s}^2+\Vert E(\tau)\Vert_{H^s}^2+\| \nabla B(\tau\|_{H^{s-1}}^2\right)\dtau\\
	\lesssim\,&\left\Vert U_{0}\right\Vert _{H^s}^{2}+\mathcal{E}_s(t)^2\mathcal{D}_s(t)^2+\Big(\mathcal N_0(t)+\mathcal N_0^2(t)+\mathcal N_1(t)\Big)\mathcal D_s^2(t).   
 \end{aligned}
\end{equation}
This shows that \eqref{Estimate_induction} holds for $\ell=0$. 

Next, assume that \eqref{Estimate_induction} is true for $ \ell  $ and we prove that  \eqref{Estimate_induction} still holds for $ \ell+1 $. Choosing $ \alpha=\ell+1 $ in Proposition \ref{Estimate_E_D_lemma}, summing \eqref{Estimate_Uk_lemma} for $ \ell+1\le k \le s $ we have
\begin{equation}
    \label{induction_l=1_u_E}
    \begin{aligned}
        &\st\left( 1+\tau\right) ^{\ell+1 }\Vert \nabla^{\ell+1} U\left( \tau\right) \Vert
		_{H^{s-(\ell+1)}}^{2}\\
  &+2\int_{0}^{t}\left( 1+\tau\right) ^{\ell+1}\left(\mu\Vert \nabla^{{\ell+2}} u(\tau)\Vert_{H^{s-(\ell+1)}}^2+\Vert \nabla^{\ell+1} E(\tau)\Vert_{H^{s-(\ell+1)}}^2\right)\dtau   \\
	\lesssim\, &\left\Vert U_{0}\right\Vert _{H^s}^{2}+(\ell+1) \int_{0}^{t}\left(
	1+\tau\right) ^{\ell} \Vert \nabla^{\ell+1}U\left( \tau\right) \Vert _{H^{s-(\ell+1)}}^{2}\dtau\\
	&+\mathcal N_0(t)\int_{0}^{t}(1+\tau)^{-N/2+\ell+1} \Vert \nabla^{\ell+1}{U}(\tau)\Vert_{H^{s-(\ell+1)}}^2\dtau\\
	&+\mathcal N_0^2(t)\int_{0}^{t}(1+\tau)^{-N+\ell+1} \Vert \nabla^{\ell+1}{U}(\tau)\Vert_{H^{s-(\ell+1)}}^2\dtau\\
	&+\mathcal N_1(t)\int_{0}^{t}(1+\tau)^{-N/2-1/2+\ell+1} \Vert \nabla^{\ell+1}{U}(\tau)\Vert_{H^{s-(\ell+1)}}^2\dtau,
  \end{aligned}
\end{equation}
where we have used the estimate
\begin{equation}
\Vert \nabla^{\ell+1}U_{0}\Vert _{H^{s-\ell-1}}^2\lesssim  \left\Vert U_{0}\right\Vert _{H^{s}}^2. 
\end{equation}
Our goal now is to estimate the second term on the right-hand side of \eqref{induction_l=1_u_E} using induction on $\ell$. 
First, we have since (by the induction assumption) \eqref{Estimate_D_E_2} is true for $\ell$, 
\begin{equation}\label{B_Estimate_Induction}
\begin{aligned}
    \int_{0}^{t}\left(
	1+\tau\right) ^{\ell} \Vert \nabla^{\ell+1}B\left( \tau\right) \Vert _{H^{s-(\ell+1)}}^{2}\dtau \lesssim&\,  \|U_0\|_{H^s}^2+\mathcal{E}_s(t)^2\mathcal{D}_s(t)^2\\
	&+\Big(\mathcal N_0(t)+\mathcal N_0^2(t)+\mathcal N_1(t)\Big)\mathcal D_s^2(t).
	\end{aligned}
\end{equation}
Next, using the estimate 
\begin{equation}\label{Sobolev_Estimate}
    \|\nabla^{k_1} f\|_{H^{s}}\lesssim \|\nabla^{k_2} f\|_{H^{s+k_1-k_2}},\qquad k_1\geq k_2
\end{equation}
together with the induction assumption \eqref{Estimate_D_E} (which we assume to hold for $\ell$), we have 
\begin{equation}\label{E_Estimate_Induction}
\begin{aligned}
    &\int_{0}^{t}\left(
	1+\tau\right) ^{\ell} \Vert \nabla^{\ell+1}E\left( \tau\right) \Vert _{H^{s-(\ell+1)}}^{2}\dtau\\\lesssim &\,\int_{0}^{t}\left(
	1+\tau\right) ^{\ell} \Vert \nabla^{\ell}E\left( \tau\right) \Vert _{H^{s-\ell}}^{2}\dtau\\
 \lesssim&\, \|U_0\|_{H^s}^2+\mathcal{E}_s(t)^2\mathcal{D}_s(t)^2
 +\Big(\mathcal N_0(t)+\mathcal N_0^2(t)+\mathcal N_1(t)\Big)\mathcal D_s^2(t).
 \end{aligned}
\end{equation}
Also using the induction assumption on \eqref{Estimate_D_E}, together with the embedding $H^{s-(\ell+1)}\hookrightarrow H^{s-\ell}$, we have 
\begin{equation}\label{u_Estimate_Induction}
\begin{aligned}
    &\int_{0}^{t}\left(
	1+\tau\right) ^{\ell} \Vert \nabla^{\ell+1}u\left( \tau\right)\Vert _{H^{s-(\ell+1)}}^{2}\dtau\\
   \lesssim &\,   \int_{0}^{t}\left(
	1+\tau\right) ^{\ell}\Vert \nabla^{\ell+1}u\left( \tau\right) \Vert _{H^{s-\ell}}^{2}\dtau\\
 &\lesssim \, \|U_0\|_{H^s}^2+\mathcal{E}_s(t)^2\mathcal{D}_s(t)^2
 +\Big(\mathcal N_0(t)+\mathcal N_0^2(t)+\mathcal N_1(t)\Big)\mathcal D_s^2(t).
 \end{aligned}
\end{equation}
Hence, collecting \eqref{B_Estimate_Induction}, \eqref{E_Estimate_Induction} and \eqref{u_Estimate_Induction}, we arrive at 
\begin{equation}
\begin{aligned}
    \int_{0}^{t}\left(
	1+\tau\right) ^{\ell} \Vert \nabla^{\ell+1}U\left( \tau\right) \Vert _{H^{s-(\ell+1)}}^{2}\dtau\lesssim&\,  \|U_0\|_{H^s}^2++\mathcal{E}_s(t)^2\mathcal{D}_s(t)^2\\
	&+\Big(\mathcal N_0(t)+\mathcal N_0^2(t)+\mathcal N_1(t)\Big)\mathcal D_s^2(t).
	\end{aligned}
\end{equation}
Consequently, we deduce from \eqref{induction_l=1_u_E}
\begin{equation}
    \label{induction_l=1_u_E_2}
    \begin{aligned}
        &\st\left( 1+\tau\right) ^{\ell+1 }\big\Vert \nabla^{\ell+1} U\left( \tau\right) \big\Vert
		_{H^{s-(\ell+1)}}^{2}\\
  &\quad+2\int_{0}^{t}\left( 1+\tau\right) ^{\ell+1}\left(\mu\Vert \nabla^{{\ell+2}} u(\tau)\Vert_{H^{s-(\ell+1)}}^2+\Vert \nabla^{\ell+1} E(\tau)\Vert_{H^{s-(\ell+1)}}^2\right)\dtau   \\
	\lesssim\, &\|U_0\|_{H^s}^2+\mathcal{E}_s(t)^2\mathcal{D}_s(t)^2+\Big(\mathcal N_0(t)+\mathcal N_0^2(t)+\mathcal N_1(t)\Big)\mathcal D_s^2(t)\\
	&+\mathcal N_0(t)\int_{0}^{t}(1+\tau)^{-N/2+\ell+1} \Vert \nabla^{\ell+1}{U}(\tau)\Vert_{H^{s-(\ell+1)}}^2\dtau\\
	&+\mathcal N_0^2(t)\int_{0}^{t}(1+\tau)^{-N+\ell+1} \Vert \nabla^{\ell+1}{U}(\tau)\Vert_{H^{s-(\ell+1)}}^2\dtau\\
	&+\mathcal N_1(t)\int_{0}^{t}(1+\tau)^{-N/2-1/2+\ell+1} \Vert \nabla^{\ell+1}{U}(\tau)\Vert_{H^{s-(\ell+1)}}^2\dtau. 
  \end{aligned}
\end{equation}
By the same method,  we take $\alpha=\ell+1$ in \eqref{Estimate_time_integral_B_lemma} and adding over $k$, with $1\leq \ell\leq k\leq s-1$, 
we obtain 
\begin{equation}\label{Estimate_main_L_2_l}
\begin{aligned}
&\int_0^t(1+t)^{\ell+1}\Vert \nabla^{\ell+1} B\Vert_{H^{s-\ell-1}}^2\dtau\\
\lesssim\,&\Vert  U_0\Vert_{H^s}^2
+(1+t)^{\ell+1}\Vert \nabla^{\ell+1} U\Vert_{H^{s-\ell-1}}^2
+  \int_{0}^{t}(1+\tau)^{\ell+1} \Vert \nabla^{\ell+1}E \Vert _{H^{s-\ell-1}}^2 \dtau\\
&+\mathcal{N}_0^2(t)\int_{0}^{t}(1+\tau)^{-N+\ell+1}\Vert
\nabla^{\ell} U(\tau)\Vert _{H^{s-\ell-1}}^2\dtau
\\
&+(\ell+1)\int_0^t(1+\tau)^{\ell}\Vert \nabla^{\ell+1} B\Vert_{H^{s-\ell-1}}^2\; \dtau.
\end{aligned}
\end{equation}
Using \eqref{B_Estimate_Induction}, we arrive at 
\begin{equation}\label{Estimate_main_L_2_2}
\begin{aligned}
&\int_0^t(1+t)^{\ell+1}\Vert \nabla^{\ell+1} B\Vert_{H^{s-\ell-1}}^2\dtau\\
\lesssim\,&\Vert  U_0\Vert_{H^s}^2+\mathcal{E}_s(t)^2\mathcal{D}_s(t)^2
	+\Big(\mathcal N_0(t)+\mathcal N_0^2(t)+\mathcal N_1(t)\Big)\mathcal D_s^2(t)\\
&+(1+t)^{\ell+1}\Vert \nabla^{\ell+1} U\Vert_{H^{s-\ell-1}}^2
+  \int_{0}^{t}(1+\tau)^{\ell+1} \Vert \nabla^{\ell+1}E \Vert _{H^{s-\ell-1}}^2 \dtau\\
&+\mathcal{N}_0^2(t)\int_{0}^{t}(1+\tau)^{-N+\ell+1}\Vert
\nabla^{\ell} U(\tau)\Vert _{H^{s-\ell-1}}^2\dtau.
\end{aligned}
\end{equation}
As above, for some $\hat{\lambda}>0$ small enough, we have by taking \eqref
{induction_l=1_u_E}+$\hat{\lambda}$\eqref{Estimate_main_L_2_l}, 
\begin{equation}
    \begin{aligned}
&\st\left( 1+\tau\right) ^{\ell+1 }\Vert \nabla^{\ell+1} U\left( \tau\right) \Vert
		_{H^{s-(\ell+1)}}^{2}\\
&+\int_{0}^{t}\left( 1+\tau\right) ^{\ell+1 }\left(\Vert \nabla^{\ell+2} u(\tau)\Vert_{H^{s-\ell-1}}^2+\Vert \nabla^{\ell+1} E(\tau)\Vert_{H^{s-\ell-1}}^2+\Vert \nabla^{\ell+1} B\Vert_{H^{s-\ell-1}}^2\right)\dtau   \\
\lesssim &\, \|U_0\|_{H^s}^2+\mathcal{E}_s(t)^2\mathcal{D}_s(t)^2+\Big(\mathcal N_0(t)+\mathcal N_0^2(t)+\mathcal N_1(t)\Big)\mathcal D_s^2(t)
&\\
	&+\mathcal N_0(t)\int_{0}^{t}(1+\tau)^{-N/2+\ell+1} \Vert \nabla^{\ell+1}{U}(\tau)\Vert_{H^{s-(\ell+1)}}^2\dtau\\
	&+\mathcal N_0^2(t)\int_{0}^{t}(1+\tau)^{-N+\ell+1} \Vert \nabla^{\ell+1}{U}(\tau)\Vert_{H^{s-(\ell+1)}}^2\dtau\\
	&+\mathcal N_1(t)\int_{0}^{t}(1+\tau)^{-N/2-1/2+\ell+1} \Vert \nabla^{\ell+1}{U}(\tau)\Vert_{H^{s-(\ell+1)}}^2\dtau\\
	\lesssim&\,\|U_0\|_{H^s}^2+\mathcal{E}_s(t)^2\mathcal{D}_s(t)^2+\Big(\mathcal N_0(t)+\mathcal N_0^2(t)+\mathcal N_1(t)\Big)\mathcal D_s^2(t)+\rm{R}_1^{\ell+1}+ \rm{R}_2^{\ell+1}+\rm{R}_3^{\ell+1},
 \end{aligned}
\end{equation}
where the hidden constant depends on  $\hat{\lambda}$ and as in \eqref{induction_l=0}, the terms  $\rm{R}_1^{\ell+1}$, $\rm{R}_2^{\ell+1}$ and $\rm{R}_3^{\ell+1}$ denote, respectively, the three-integral terms appearing on the right-hand side of the preceding estimate.

\noindent
Using the estimate 
\begin{equation}
\begin{aligned}
   & \mathrm{R}_1^{\ell+1}+ \mathrm{R}_2^{\ell+1}+\mathrm{R}_3^{\ell+1}\\
    \lesssim&\, \left(\mathcal N_0(t)+\mathcal N_0^2(t)+\mathcal N_1(t)\right)\int_{0}^{t}(1+\tau)^{-N/2+\ell+1} \Vert \nabla^{\ell+1}{U}(\tau)\Vert_{H^{s-(\ell+1)}}^2\dtau,
    \end{aligned}
\end{equation}
we arrive at 
\begin{equation}\label{Estimate_U_k_l_3}
    \begin{aligned}
        &\st\left( 1+\tau\right) ^{\ell+1 }\Vert \nabla^{\ell+1} U\left( \tau\right) \Vert
		_{H^{s-(\ell+1)}}^{2}\\
&+\int_{0}^{t}\left( 1+\tau\right) ^{\ell+1 }\left(\Vert \nabla^{\ell+2} u(\tau)\Vert_{H^{s-\ell-1}}^2+\Vert \nabla^{\ell+1} E(\tau)\Vert_{H^{s-\ell-1}}^2\right)\dtau   \\
&+\int_{0}^{t}\left( 1+\tau\right) ^{\ell+1 }\Vert \nabla^{\ell+1} B\Vert_{H^{s-\ell-1}}^2\dtau \\
\lesssim &\, \|U_0\|_{H^s}^2+\mathcal{E}_s(t)^2\mathcal{D}_s(t)^2+\Big(\mathcal N_0(t)+\mathcal N_0^2(t)+\mathcal N_1(t)\Big)\mathcal D_s^2(t)\\
&+ \left(\mathcal N_0(t)+\mathcal N_0^2(t)+\mathcal N_1(t)\right)\int_{0}^{t}(1+\tau)^{-N/2+\ell+1} \Vert \nabla^{\ell+1}{U}(\tau)\Vert_{H^{s-(\ell+1)}}^2\dtau.
    \end{aligned}
\end{equation}
Since $N\geq 2$, we have 
\begin{equation}\label{N_1_D_s}
\begin{aligned}
&\int_{0}^{t}(1+\tau)^{-N/2+\ell+1} \Vert \nabla^{\ell+1}{U}(\tau)\Vert_{H^{s-(\ell+1)}}^2\dtau\\
\lesssim&\, \int_{0}^{t}(1+\tau)^{\ell} \Vert \nabla^{\ell+1}{U}(\tau)\Vert_{H^{s-(\ell+1)}}^2\dtau
\lesssim\, \mathcal D_s^{2}(t), 
\end {aligned}
\end{equation}
consequently, using \eqref{N_1_D_s}, we obtain from \eqref{Estimate_U_k_l_3}
and 
\eqref{Estimate_main_L_2_l}, 
\begin{equation}\label{Estimate_U_k_l_4}
    \begin{aligned}
&\st\left( 1+\tau\right) ^{\ell+1 }\Vert \nabla^{\ell+1} U\left( \tau\right) \Vert
		_{H^{s-(\ell+1)}}^{2}\\
&+\int_{0}^{t}\left( 1+\tau\right) ^{\ell+1 }\left(\Vert \nabla^{\ell+2} u(\tau)\Vert_{H^{s-\ell-1}}^2+\Vert \nabla^{\ell+1} E(\tau)\Vert_{H^{s-\ell-1}}^2+\Vert \nabla^{\ell+1} B\Vert_{H^{s-\ell-1}}^2\right)\dtau   \\
\lesssim &\, \|U_0\|_{H^s}^2+\mathcal{E}_s(t)^2\mathcal{D}_s(t)^2+\Big(\mathcal N_0(t)+\mathcal N_0^2(t)+\mathcal N_1(t)\Big)\mathcal D_s^2(t). 
\end{aligned}
\end{equation}
Thus, \eqref{Estimate_induction} holds for $\ell+1$. 
This finishes the proof of Proposition \ref{Proposition_H_s_Estimate}.

\subsection{Time-decay estimate--proof of \eqref{Weighted_estimate_mass}}\label{Subsection_Decay}

In this section, we prove the desired estimate \eqref{Weighted_estimate_mass} for the time-weighted functional $\mathcal{M}_s(t)$. This estimate will be combined with the energy estimates obtained above to prove
Theorem \ref{Theorem_Global_existence}. 
\begin{proposition}\label{proposition_Decay}
    It holds that 
    \begin{equation}\label{Weighted_estimate_mass_2}
	\mathcal M_s(t)\lesssim (\|U_0\|_{L^1}+\|U_0\|_{H^s})+\mathcal M_s(t)\mathcal N_0(t)+\mathcal M_s^2(t)+\mathcal M_s^2(t)\mathcal N_0(t).
\end{equation}
\end{proposition}
\begin{proof}
 By Duhamel's formula, the solution of \eqref{Main_System} can be written as 
\begin{equation}\label{Duhamel_formula}
	U(t)=e^{t\mathcal{L}}U_0+\int_0^te^{(t-\tau)\mathcal{L}}\mathbb{P}(f_1(\tau)+f_2(\tau)+f_3(\tau)+f_4(\tau))\dtau,
\end{equation}
where
\begin{equation}
    \begin{aligned}
        &f_1= (-E\times B,0,0), \qquad f_2=(-(u\times B)\times B,0,0)\\
      &f_3=(\nabla\cdot(u\otimes u), 0, 0),\qquad f_4=(0,-u\times B,0)  
    \end{aligned}
\end{equation}
Applying $ \nabla^k $ for $ 0\le k \le s $ and taking $ L^2 $ norm we obtain
\begin{equation}\label{Estimate_M}
    \begin{aligned}
	\|\nabla^k U\|_{L^2} \lesssim\,& \|e^{t\mathcal{L}}\nabla^{k}U_0\|_{L^2}  +\int_0^t\|e^{(t-\tau)\mathcal{L}}\nabla^{k}f_1(\tau)\|_{L^2}\dtau+\int_0^t\|e^{(t-\tau)\mathcal{L}}\nabla^{k}f_2(\tau)\|_{L^2}\dtau\\
	&+\int_0^t\|e^{(t-\tau)\mathcal{L}}\nabla^{k}f_3(\tau)\|_{L^2}+\int_0^t\|e^{(t-\tau)\mathcal{L}}\nabla^{k}f_4(\tau))\|_{L^2}\dtau \\
:=\,&\rm{I}_0+\rm{I}_1+\rm{I}_2+\rm{I}_3+\rm{I}_4.
\end{aligned}
\end{equation}
Our goal now is to estimate the terms $\rm{I}_i,\, i=1,\dots,4$. 

For $\rm{I}_0 $, applying the linear estimate in Lemma \ref{MHD_linear_stability} we get 
\begin{equation}\label{estimate_I0}
	{\rm{I}}_0\le C(1+t)^{-N/4-k/2}\|U_0\|_{L^1}+e^{-ct}\|\nabla^{k}U_0\|_{L^2}.
\end{equation}
For the terms $ \rm{I}_1,\dots, \rm{I}_4 $, we split the integral into two parts over the intervals $[0,t/2]$ and $[t/2,t]$  and apply Lemmas  \ref{MHD_linear_stability}, \ref{Lemma_E} and  \ref{Improved_Estimate_B_Lemma}  with different $ U_0 $ and $V_0$. For example, when dealing with $ \rm{I}_1 $, using Lemma \ref{MHD_linear_stability} with $ U_0=f_1 $ and $ U_0=\nabla^{k}f_1 $ for each part, respectively. 
Hence, we have 
(in this case $ V_0=0$ in Lemma \ref{MHD_linear_stability}) 
\begin{equation}
\begin{aligned}  
{\rm{I}}_1\lesssim&\int_0^{t/2}(1+t-\tau)^{-N/4-k/2}\|E\times B\|_{L^1}\dtau+\int_{t/2}^t(1+t-\tau)^{-N/4}\|\nabla^{k}(E\times B)\|_{L^1} \dtau. \\
	\lesssim &\int_0^{t/2}(1+t-\tau)^{-N/4-k/2}\|E\|_{L^2}\|B\|_{L^2}\dtau\\
	&+\int_{t/2}^t(1+t-\tau)^{-N/4}(\|E\|_{L^2}\|\nabla^{k} B\|_{L^2}+\|\nabla^{k}E\|_{L^2}\|B\|_{L^2})\dtau\\
    \end{aligned}
\end{equation}
The fastest decay rate of the term  $\|E\|_{L^2}\lesssim (1+t)^{-N/4-1/2}$ is useful here to avoid the logarithm loss in the first integral in the two-dimensional case.  Hence, we have 
\begin{equation}
 \label{estimate_I1}
\begin{aligned} 
	{\rm{I}}_1\lesssim &\int_0^{t/2}(1+t-\tau)^{-N/4-k/2}\|E(\tau)\|_{L^2}\|B(\tau)\|_{L^2}\dtau\\
    &+\int_{t/2}^t(1+t-\tau)^{-N/4}\|U(\tau)\|_{L^2}\|\nabla^{k} U(\tau)\|_{L^2}\dtau \\
	\lesssim &\,\mathcal M_s^2(t)\int_0^{t/2}(1+t-\tau)^{-N/4-k/2}(1+\tau)^{-N/2-1/2}\dtau \\
	&+ \mathcal M_s^2(t)\int_{t/2}^t(1+t-\tau)^{-N/4}(1+\tau)^{-N/2-k/2}\dtau\\
	\lesssim &(1+t)^{-N/4-k/2}\mathcal M_s^2(t),
    \qquad 
 \end{aligned}
\end{equation}
where we applied Lemma \ref{Integral_lemma} and  used the estimate 
\begin{equation}
    \|\nabla^\ell U(\tau)\|_{L^2}\le C(1+\tau)^{-N/4-\ell/2}\mathcal M_s(t) \quad \text{for}\quad   0\le \ell \le s 
\end{equation}
 together with 
 \begin{equation}
     \|E(\tau)\|_{L^2}\lesssim (1+\tau)^{-N/4-1/2}\mathcal M_s(t).
 \end{equation}
For $ \rm{I}_2 $, similar to $ \rm{I}_1 $, using the fact that
\begin{equation}
\begin{aligned}
	&\|\nabla^\ell( (u\times B)\times B)\|_{L^1}\le C\|U\|_{L^\infty}\|U\|_{L^2}\|\nabla^\ell U\|_{L^2},\\
	 & \|\nabla^\ell U\|_{L^2}\le C(1+t)^{-N/4-\ell/2}\mathcal M_s(t),\quad \,0\le \ell \le s,
	\end{aligned}
\end{equation}
and 
\begin{equation}
	\|U\|_{L^\infty}\le C(1+t)^{-N/2}\mathcal N_0(t),
\end{equation}
 one has
\begin{equation}
\begin{aligned}
	{\rm{I}}_2\lesssim&\,\int_0^{t/2}(1+t-\tau)^{-N/4-k/2}\|(u\times B)\times B\|_{L^1}\dtau\\
	&+\int_{t/2}^t(1+t-\tau)^{-N/4}\|\partial ^k((u\times B)\times B)\|_{L^1} \dtau \\
	\lesssim&\,\int_0^{t/2}(1+t-\tau)^{-N/4-k/2}\|U\|_{L^\infty}\|U\|_{L^2}^2\dtau\notag\\
	&+\int_{t/2}^t(1+t-\tau)^{-N/4}\|U\|_{L^\infty}\|U\|_{L^2}\|\nabla^{k} U\|_{L^2} \dtau. 
	\end{aligned}
\end{equation}
This gives 
\begin{equation}\label{estimate_I2}
\begin{aligned}
	{\rm{I}}_2\lesssim \,& \mathcal N_0(t)\mathcal M_s^2(t)\int_0^{t/2}(1+t-\tau)^{-N/4-k/2}(1+\tau)^{-N}\dtau\\
	&+\mathcal N_0(t)\mathcal M_s^2(t)\int_{t/2}^t(1+t-\tau)^{-N/4}(1+\tau)^{-N-k/2} \dtau \\
	\lesssim&\,\mathcal N_0(t)\mathcal M_s^2(t)(1+t)^{-N/4-k/2}.
	\end{aligned}
\end{equation}
The estimate of $ \rm{I}_3 $ is slightly different. For the integral from $ 0 $ to $ t/2 $, we choose $ U_0=f_3 $ in Lemma \ref{MHD_linear_stability}. However for the rest part of integral, we choose $ U_0=\nabla^{k}f_3 $:
\begin{equation}
\begin{aligned}
		{\rm{I}_3}\lesssim\,&\int_0^{t/2}(1+t-\tau)^{-N/4-k/2}\|f_3(\tau)\|_{L^1}\dtau+\int_{t/2}^t\| e^{(t-\tau)\mathcal L}\nabla^{k}f_3(\tau)\|_{L^2} \dtau\nonumber \\
		\lesssim\,& \int_0^{t/2}(1+t-\tau)^{-N/4-(k+1)/2}\|u\otimes u\|_{L^1}\dtau+\int_{t/2}^t(1+t-\tau)^{-N/4-1/2}\|\nabla^{k}(u\otimes u)\|_{L^1} \dtau\nonumber \\
		\lesssim\,& \int_0^{t/2}(1+t-\tau)^{-N/4-(k+1)/2}\|U\|_{L^2}^2\dtau+\int_{t/2}^t(1+t-\tau)^{-N/4-1/2}\|U\|_{L^2}\|\nabla^{k}U\|_{L^2} \dtau\nonumber 
		\end{aligned}
\end{equation}
Applying Lemma \ref{Integral_lemma}, we obtain 
\begin{equation}\label{estimate_I3}
\begin{aligned}
		{\rm{I}_3}\lesssim\,&\mathcal M_s^2(t)\int_0^{t/2}(1+t-\tau)^{-N/4-(k+1)/2}(1+\tau)^{-N/2}\dtau\nonumber\\
		+&\mathcal M_s^2(t)\int_{t/2}^t(1+t-\tau)^{-N/4-1/2}(1+\tau)^{-N/2-k/2}\dtau \\
		\lesssim\,&\mathcal M_s^2(t)(1+t)^{-N/4-k/2}.
		\end{aligned}
\end{equation}
The estimate of $\rm{I}_4$ will also experience a logarithm loss in the two-dimensional case if we proceed as in the estimate of $\rm{I}_1$, where we exploit the improved decay of $E$ in the nonlinearity $E\times B$. That does not work here since there is no faster decay for any of the components in the source term $u\times B$ to recover the loss as in the case of $\rm{I}_1$. However, the structure of $f_4$ helps eliminate this loss. Recall that $f_4=(0,-u\times B, 0)$. Hence, the forcing is entirely in the electric-field equation.

The estimate of $ \rm{I}_4 $ goes like the previous ones, but noting that we take $ f_4 =(0,-u\times B, 0)$ acts  only on the electric field equation, we can apply Lemma \ref{Lemma_E} and Lemma \ref{Improved_Estimate_B_Lemma} with the data $V_0=(-u\times B,0)$,
we have
\begin{equation}
\begin{aligned}	{\rm{I}}_4\lesssim&\,\int_0^{t/2}(1+t-\tau)^{-N/4-1/2-k/2}\|u\times B\|_{L^1}\dtau\\
&+\int_{t/2}^t(1+t-\tau)^{-N/4-1/2}\|\nabla^{k}(u\times B)\|_{L^1} \dtau 
	+\int_0^te^{-c(t-\tau)}\|\nabla^{k}(u\times B)\|_{L^2}\dtau\\
	\lesssim&\, \int_0^{t/2}(1+t-\tau)^{-N/4-1/2-k/2}\|U\|^2_{L^2}\dtau
    +\int_{t/2}^t(1+t-\tau)^{-N/4-1/2}\|U\|_{L^2}\|\nabla^{k}U\|_{L^2} \dtau \\
	 &+\int_0^te^{-c(t-\tau)}\|U\|_{L^\infty}\|\nabla^{k}U\|_{L^2}\dtau. 
	\end{aligned}
\end{equation}
Consequently, we obtain 
	\begin{equation}\label{estimate_I4}
\begin{aligned}
	{\rm{I}}_4	\lesssim&\, \mathcal M_s^2(t)\int_0^{t/2}(1+t-\tau)^{-N/4-1/2-k/2}(1+\tau)^{-N/2} \dtau\\
		&+\mathcal M_s^2(t) \int_{t/2}^t(1+t-\tau)^{-N/4-1/2}(1+\tau)^{-N/2-k/2}\dtau \\
		& +\mathcal{N}_0(t)\mathcal M_s(t)\int_0^te^{-c(t-\tau)}(1+\tau)^{-3N/4-k/2}\dtau\\
		\lesssim&\,(1+t)^{-N/4-k/2}(\mathcal M_s(t)^2+\mathcal N_0(t)\mathcal M_s(t)),
		\end{aligned}
\end{equation}
for $N=3$. For $N=2$, the first two integrals in \eqref{estimate_I4} are treated differently. Indeed, we have  
\begin{equation}
\begin{aligned}
    \int_0^{t/2}(1+t-\tau)^{-N/4-1/2-k/2}(1+\tau)^{-N/2} \dtau=&\,\int_0^{t/2}(1+t-\tau)^{-1-k/2}(1+\tau)^{-1} \dtau\\
    \lesssim&\, (1+t)^{-1-k/2}\log(2+t)\\
    \lesssim&\, (1+t)^{-1/2-k/2}. 
    \end{aligned}
\end{equation}
Similarly, 
\begin{equation}
    \begin{aligned}
        \int_{t/2}^t(1+t-\tau)^{-N/4-1/2}(1+\tau)^{-N/2-k/2}\dtau\lesssim&\, (1+t)^{-1-k/2}\int_{t/2}^t(1+\tau)^{-N/2-k/2}\dtau\\
    \lesssim&\, (1+t)^{-1-k/2}\log(2+t)\\
    \lesssim&\, (1+t)^{-1/2-k/2}.    \end{aligned}
\end{equation}
Consequently, we have for $N=2,\, 3$
\begin{equation}\label{estimate_I4_2}
    {\rm{I}_4}\lesssim \, (1+t)^{-N/4-k/2}(\mathcal M_s(t)^2+\mathcal N_0(t)\mathcal M_s(t)). 
\end{equation}

Now, we derive the estimate of $E$. Recall that $E$ satisfies the equation 
\begin{equation}\label{E_Nonl_Eqs}  \partial_t E+E=\nabla\times B-u\times B. 
\end{equation}
Applying $\nabla^k$ to \eqref{E_Nonl_Eqs} and using the Duhamel formula, we obtain 
\begin{equation}\label{E_Main_Est}
\begin{aligned}
    \|\nabla^k E(t)\|_{L^2}\lesssim&\, e^{-t}\|\nabla^k E_0\|_{L^2}+\int_0^t e^{-(t-\tau)}\| \nabla^{k+1} B)(\tau)\|_{L^2}\dtau\\
    &+\int_0^t e^{-(t-\tau)}\|\nabla^k (u\times B)(\tau)\|_{L^2}\dtau. 
    \end{aligned}
\end{equation}
First, we have from \eqref{M_s} and for all $0\leq k\leq s-1$,
\begin{equation}\label{B_k_E_Est}
    \| \nabla^{k+1} B)(\tau)\|_{L^2}\lesssim \mathcal{M}_s(t)(1+\tau)^{-N/4-1/2-k/2}. 
\end{equation}
Next, we have the estimate (see \eqref{First_inequaliy_Guass})
\begin{equation}\label{Prod_Est_E}
\begin{aligned}
    \|\nabla^k (u\times B)(\tau)\|_{L^2}\lesssim &\,\|u\|_{L^\infty}\|\nabla^k B\|_{L^2}+\|B\|_{L^\infty}\|\nabla^k u\|_{L^2}\\
    \lesssim &\, \mathcal{N}_0(t)(1+\tau)^{-N/2}\mathcal{M}_s(t) (1+\tau)^{-N/4-k/2}\\
    \lesssim &\,\mathcal{N}_0(t)\mathcal{M}_s(t)(1+\tau)^{-3N/4-k/2}. 
    \end{aligned}
\end{equation}
Plugging \eqref{B_k_E_Est} and \eqref{Prod_Est_E} into \eqref{E_Main_Est}, we obtain for $0\leq k\leq s-1$, 
\begin{equation}\label{E_final_Estimate}
\begin{aligned}
     \|\nabla^k E(t)\|_{L^2}\lesssim&\, (1+t)^{-N/4-1/2-k/2}(\|\nabla^k E_0\|_{L^2}+\mathcal{M}_s(t)+\mathcal{N}_0(t)\mathcal{M}_s(t)). 
     \end{aligned}
\end{equation}
Collecting the estimates  \eqref{estimate_I1}, \eqref{estimate_I2}, \eqref{estimate_I3},  \eqref{estimate_I4_2} and \eqref{E_final_Estimate} we obtain \eqref{Weighted_estimate_mass_2} and hence  finish the proof of Proposition \ref{proposition_Decay}. 
\end{proof}
\subsection{Closing the estimates--Proof of Theorem \ref{Theorem_Global_existence}}\label{Subsection_Closing}
Using the estimates \eqref{Weighted_estimate} and \eqref{Weighted_estimate_mass}
together with  Lemma \ref{Lemma_N_i_M_s}, we obtain 
\begin{subequations}  
\begin{equation}
\TE_s^{2}\left( t\right) +\mathcal D_s^{2}\left( t\right) \lesssim  \|U_0\|_{H^s}^2+\mathcal{E}_s(t)^2\mathcal{D}_s(t)^2+(\mathcal M_s(t)+\mathcal M_s^2(t))\mathcal D_s^2(t),  \label{Weighted_estimate_1}
\end{equation}
and 
\begin{equation}\label{Weighted_estimate_mass_1}
	\mathcal M_s(t)\lesssim (\|U_0\|_{L^1}+\|U_0\|_{H^s})+\mathcal M_s(t)^2+\mathcal M_s^2(t)+\mathcal M_s^3(t).
\end{equation}
\end{subequations}
By setting 
\begin{equation}
\mathcal{Y}_s(t)=\mathcal {E}_s(t)+\mathcal {D}_s(t)+\mathcal {M}_s(t),
\end{equation}
we arrive at the inequality 
\begin{equation}
\mathcal{Y}_s(t)^2\lesssim (\|U_0\|_{L^1}+\|U_0\|_{H^s})^2+\mathcal{Y}_s(t)^3+\mathcal{Y}_s(t)^4.
\end{equation}
for which we can deduce (by applying Lemma \ref{Lemma_Stauss}) that 
\begin{equation}
\mathcal{Y}_s(t)\lesssim  \|U_0\|_{L^1}+\|U_0\|_{H^s}
\end{equation}
provided that $\Lambda=\|U_0\|_{L^1}+\|U_0\|_{H^s}$ is small enough, say $\Lambda\leq \delta_0$. This gives the desired a priori estimates of the solution, by which we can continue a   unique local solution to be global in time. The obtained solution verifies the decay estimates \eqref{Decay_Main} since we have shown that $\mathcal{M}_s(t)\lesssim \Lambda $. This finishes the proof of Theorem \ref{Theorem_Global_existence}.

\subsection*{Acknowledgements}
During the preparation of this manuscript, the author used ChatGPT (OpenAI) solely for language editing and improvement of the clarity and readability of the manuscript. All mathematical contents, arguments, and results were developed and verified by the author. 
\subsection*{Statements and Declarations} The author declares that he has no conflict of interest. 

 \subsection*{Data availability}
No data sets were generated or analyzed during the current study.

  
	\end{document}